\documentclass[reqno]{amsart}

\usepackage[utf8]{inputenc}
\usepackage{amssymb}
\usepackage{graphicx}
\usepackage{latexsym}
\usepackage{stmaryrd}
\usepackage{enumitem}
\usepackage[bookmarks]{hyperref}
\usepackage{hyperref}
\usepackage{tikz-cd}
\usetikzlibrary{arrows}
\usepackage{relsize}
\usepackage{exscale}
\usepackage{xifthen}
\usepackage{mathrsfs}
\usepackage{multicol}

\newcommand{\F}{\Gamma}
\newcommand{\modG}{\vert G\vert_{\Theta}}
\newcommand{\minG}{\vert G\vert_{\F}}
\newcommand{\VG}{\vert G\vert_\Omega}
\newcommand{\finC}[1]{{\F_{#1}}}
\newcommand{\minC}{\cC^{\F}}
\newcommand{\toC}[2]{
	\ifthenelse{\isempty{#1}}
		{\varphi_{#2}}
		{\varphi_{#2}(#1)}
}

\newcommand{\crit}{\normalfont\text{crit}}
\newcommand{\invLim}{\varprojlim}
\newcommand{\cb}{\mathfrak{c}}
\newcommand{\eb}{\mathfrak{f}}

\newcommand{\id}{\normalfont\text{id}}
\newcommand{\rest}{\upharpoonright}

\DeclareMathOperator{\medcup}{\mathsmaller{\bigcup}}
\newcommand{\Tt}{\Theta_{\normalfont\text{t}}}
\newcommand{\St}{S_{\normalfont\text{t}}}
\newcommand{\Ft}{\mathcal{T}_{<\aleph_0}^{\,\normalfont\text{t}}}
\newcommand{\Pt}[1]{P_{#1}^{\,\normalfont\text{t}}}
\newcommand{\vS}{\vec{S}}
\newcommand{\vSt}{\vS_{\normalfont\text{t}}}

\newcommand{\comp}{com\-pac\-ti\-fi\-ca\-tion}
\newcommand{\Ocomp}{$\Omega$-\comp}
\newcommand{\Ccomp}{$\mathscr{C}$-\comp}
\newcommand{\Csys}{$\mathscr{C}$-sys\-tem}
\newcommand{\HD}{Haus\-dorff}
\newcommand{\HDcomp}{\HD\ \comp}
\newcommand{\SC}{Stone-Čech}
\newcommand{\sol}{tough}
\newcommand{\esol}{end-\sol}
\newcommand{\nsol}{non-\sol}
\newcommand{\soly}{toughness}
\newcommand{\esoly}{end-\soly}
\newcommand{\nsoly}{non-\soly}

\makeatletter

\def\calCommandfactory#1{%
   \expandafter\def\csname c#1\endcsname{\mathcal{#1}}}
\def\frakCommandfactory#1{%
   \expandafter\def\csname frak#1\endcsname{\mathfrak{#1}}}

\newcounter{ctr}
\loop
  \stepcounter{ctr}
  \edef\X{\@Alph\c@ctr}
  \expandafter\calCommandfactory\X
  \expandafter\frakCommandfactory\X
  \edef\Y{\@alph\c@ctr}
  \expandafter\frakCommandfactory\Y
\ifnum\thectr<26
\repeat

\renewcommand{\cC}{\mathscr{C}}
\renewcommand{\cD}{\mathscr{D}}
\renewcommand{\cK}{\mathscr{K}}

\newtheorem{theorem}{Theorem}[section]
\newtheorem{proposition}[theorem]{Proposition}
\newtheorem{corollary}[theorem]{Corollary}
\newtheorem{lemma}[theorem]{Lemma}
\newtheorem{obs}[theorem]{Observation}

\theoremstyle{definition}
\newtheorem{definition}[theorem]{Definition}
\newtheorem{fact}[theorem]{Fact}
\newtheorem{example}[theorem]{Example}

\theoremstyle{remark}
\newtheorem*{notation}{Notation}
\newtheorem*{convention}{Convention}
\newtheorem*{ack}{Acknowledgement}

\setenumerate{label={\normalfont (\roman*)},itemsep=0pt}

\begin{document}

\title{Ends, tangles and critical vertex sets}


\author{Jan Kurkofka}
\address{University of Hamburg, Department of Mathematics, Bundesstraße 55 (Geomatikum), 20146 Hamburg, Germany}
\email{jan.kurkofka@uni-hamburg.de, max.pitz@uni-hamburg.de}

\author{Max Pitz}



\begin{abstract}
We show that an arbitrary infinite graph $G$ can be compactified by its ends plus its critical vertex sets, where a finite set $X$ of vertices of an infinite graph is \emph{critical} if its deletion leaves some infinitely many components each with neighbourhood precisely equal to $X$.

We further provide a concrete separation system whose $\aleph_0$-tangles are precisely the ends plus critical vertex sets. Our tangle \comp\ $\minG$ is a quotient of Diestel's (denoted by $\modG$), and both use tangles to compactify a graph in much the same way as the ends of a locally finite and connected graph compactify it in its Freudenthal \comp .

Finally, generalising both Diestel's construction of $\modG$ and our construction of $\minG$, we show that $G$ can be compactified by every inverse limit of \comp s of the sets of components obtained by deleting a finite set of vertices.
Diestel's $\modG$ is the finest such \comp , and our $\minG$ is the coarsest one. Both coincide if and only if all tangles are ends. This answers two questions of Diestel.
\end{abstract}
\vspace*{-1.14cm}
\maketitle

\section{Introduction}

The ends of a locally finite, connected graph naturally compactify it in its \emph{Freudenthal \comp }~\cite{Bible,RDsBanffSurvey}. For a non-locally finite graph, however, adding its ends usually no longer suffices to compactify it.  This is where its tangles of infinite order, its $\aleph_0$-\emph{tangles}, enter the scene: Recently, Diestel~\cite{EndsAndTangles} combined Halin's notion of an \emph{end} of an infinite graph~(\cite{halin64}, from 1964) with Robertson and Seymour's notion of a \emph{tangle}~(\cite{GMX}, from 1991) as follows: He first observed that every end induces an  $\aleph_0$-tangle by orienting every finite order separation of the graph towards the side where the end lives, and then proceeded to show that adding all $\aleph_0$-tangles to an arbitrary infinite graph (possibly disconnected and not locally finite) does again suffice to compactify it, yielding the \emph{tangle \comp } $\modG$ of $G$. Here and in the following, we let $\Omega$ and $\Theta$ denote the set of ends and of $\aleph_0$-tangles of a graph $G$ respectively.

Like the Freudenthal compactification, $\modG$ has a totally disconnected remainder, i.e.\ the boundary at infinity contains no non-trivial connected components. Moreover, if $G$ is locally finite and connected, then its $\aleph_0$-tangles turn out to be precisely its ends---and the tangle \comp\ coincides with the Freudenthal \comp .

Our aim in this paper is twofold: First, we want to provide a comprehensive study of the tangle \comp\ $\modG$, as well as  other related \comp s of infinite graphs, and secondly, to apply  some of these insights in order to answer the following two questions of Diestel's paper~\cite[§6]{EndsAndTangles}:
\begin{enumerate}
\item \emph{``For which $G$ is $\modG$ the coarsest \comp\ in which its ends appear as distinct points?''}
\item \emph{``If it is not, is there a unique such [\comp ], and is there a canonical way to obtain it from $\modG$?''}
\end{enumerate}
Let us call a \comp\ of $G$ that also extends the ends in a meaningful way an \emph{\Ocomp} of $G$ (see Section~\ref{sec3} for a precise definition). Answering the first question, we shall see in Theorem~\ref{thm_Q1} that the tangle \comp\ $\modG$ is the coarsest \Ocomp\ of $G$ if and only if deleting any finite set of vertices from $G$ leaves only finitely many components, a property which we call \emph{tough}. This property turns out to be equivalent to the assertion that there are no $\aleph_0$-tangles other than the ends.

To answer the second question, we construct a new compactification $\minG$ whose remainder is formed by the ends plus the critical vertex sets of $G$ (a finite set $X\subseteq V(G)$ is \emph{critical} if its deletion leaves some infinitely many components each with neighbourhood precisely equal to $X$). We show that $\minG$ is again a tangle-type \comp, and that it can be obtained from $\modG$ as a natural quotient. Strengthening these observations considerably, we then proceed to show that for a natural class of compactifications of $G$---which we call \Ocomp s induced by a \Csys ---our newly constructed $\minG$ is the least such compactification and Diestel's $\modG$ is in fact the unique largest such compactification, see Theorem~\ref{Csystem:leastAndGreatest}. Phrased differently, this means that $\minG$ is a quotient of every \Ocomp\ induced by a \Csys\ which in turn is always a  quotient of $\modG$. In particular, we may rephrase our answer to question (i), observing that $\modG$ is the coarsest \comp\ in which its ends appear as distinct points if and only if the class of \Ocomp s induced by a \Csys\ is trivial.

This paper is organised as follows. 
In Section~\ref{sec2}, we provide details on tangles and briefly review the construction of Diestel's tangle compactification.  In Section~\ref{sec3}, we formally introduce the concept of \Ocomp s and present our answer to question (i) announced above. 

In Section~\ref{sec4}, we formally introduce \emph{critical vertex sets}, and show that every infinite graph $G$ is compactified by its ends plus its critical vertex sets, giving rise to a compactification $\minG$. Furthermore, we show that the critical vertex sets naturally partition the $\aleph_0$-tangles that are not ends (the so-called \emph{ultrafilter tangles}). This defines an equivalence relation ${\sim}$ on $\Theta$ such that $\modG/{\sim}$ is an \Ocomp\ of $G$ with the desired remainder $\Omega\sqcup\crit(G)$, where $\crit(G)$ denotes the collection of all critical vertex sets.
Notably, the number of critical vertex sets is bounded above by the cardinality of the graph's vertex set, and the number of ultrafilter tangles is bounded below by the cardinal number $\vert\crit(G)\vert\cdot 2^{\frakc}$.

In Section~\ref{sec5}, we show that $\minG$ (or equivalently: the quotient $\modG/{\sim}$) is again a tangle-type \comp . More precisely, we use critical vertex sets to explicitly describe a collection $\St$ of finite order separations of $G$ such that the $\aleph_0$-tangles \emph{of} $\St$, tangles of infinite order that only orient the separations in $\St$, correspond precisely to the ends plus critical vertex sets.
The $\aleph_0$-tangles of $\St$ differ from the original $\aleph_0$-tangles in that they do not orient all the finite order separations, just those in $\St$, and so there are significantly fewer of the new ones compared to the original ones.

Next, in Section~\ref{sec6}, we formally introduce the concept of \Csys s.
Recall that the graph-theoretic ends of a graph (i.e.\ equivalence classes of rays) correspond precisely to elements of the inverse limit of the system $\{\cC_X,\cb_{X',X},\cX\}$ where $\cX$ denotes the collection of all finite subsets of $V(G)$ directed by inclusion; where $\cC_X$ is the set of components of $G-X$ and for $X' \supseteq X$, the bonding map $\cb_{X',X}\colon\cC_{X'}\to\cC_X$ sends each component of $G-X'$ to the unique component of $G-X$ including it.
Diestel showed that the limit of the inverse system $\{\beta(\cC_X),\beta(\cb_{X',X}),\cX\}$ describes the space $\Theta$ of $\aleph_0$-tangles, where $\beta(\cC_X)$ is the \SC\ \comp\ of the discrete space $\cC_X$ and the bonding maps $\beta(\cb_{X',X})$ are provided by the \SC\ property. From this description, the inclusion $\Omega \subseteq \Theta$ is now evident. Generalising this idea, we call an inverse system $\{\alpha(\cC_X),\fraka_{X',X},\cX\}$ of \HDcomp s of the component spaces $\cC_X$ a $\cC$-\emph{system} (\emph{of} $G$) if the bonding maps $\fraka_{X',X}$ continuously extend the underlying maps $\cb_{X',X}$. As our main result of this section, Theorem~\ref{Csystem:EveryCsystemInducesOmegaComp}, we show that every \Csys\ induces an \Ocomp\ of $G$ in the way Diestel used his \Csys\ to compactify $G$. 

Finally, in Section~\ref{sec7}, we shall see that also our newly constructed compactification $\minG$ is in fact induced by a \Csys. Indeed, adding to any $\cC_X$ the critical vertex sets contained in $X$ yields a natural \HD\ \comp\ $\finC{X}$ of $\cC_X$ with finite remainder, which in turn give rise to a \Csys. We then proceed to compare the different compactifications induced by \Csys s. In particular, we show that these $\finC{X}$ form the least \Csys\ with respect to a natural partial ordering. Consequently, the \Ocomp\ $\minG$ it induces turns out to be the coarsest of its kind, whereas the tangle \comp\ $\modG$ is the finest one, Theorem~\ref{Csystem:leastAndGreatest}.
We conclude this paper by showing that $\minG$ and $\modG$ are equivalent if and only if every $\cC_X$ is finite, i.e. if and only if the graph is \sol , if and only if all $\aleph_0$-tangles are ends.

\begin{ack}
We thank Johannes Carmesin and Reinhard Diestel for interesting discussions on the topic of infinite-order tangles.
\end{ack}

\section{Reviewing Diestel's tangle compactification}
\label{sec2}

\subsection{Compactifications}

A \emph{compactification} of a topological space $X$ is an ordered pair $(K,h)$ where $K$ is a compact topological space and $h\colon X\hookrightarrow K$ is an embedding of $X$ as a dense subset of $K$.
Sometimes we also refer to $K$ as a \comp\ of $X$ if the embedding $h$ is clearly understood (e.g. if $h$ is the identity on $X$).
The space $K\setminus h[X]$ is called the \emph{remainder} of the \comp .

If $(K,h)$ and $(K',h')$ are two \comp s of $X$ we write $(K,h)\le (K',h')$ whenever there exists a continuous mapping $f\colon K'\to K$ for which the diagram
\begin{equation*}
\begin{tikzcd}[column sep=1.5em]
& X\arrow[dr, "h'", right hook->]\arrow[dl, "h"', left hook->]\\
K && K'\arrow[ll, "f"']
\end{tikzcd}
\end{equation*}
commutes.
Then $(K,h)$ is said to be \emph{coarser} than $(K',h')$, and $(K',h')$ in turn is said to be \emph{finer} than $(K,h)$.
When we want to say that $(K,h)\le (K',h')$ is witnessed by a map $f\colon K'\to K$ we write $f\colon (K',h')\ge (K,h)$ for short.
If there exists a homeomorphism $f\colon (K',h')\ge (K,h)$, then we say that the two \comp s $(K,h)$ and $(K',h')$ of $X$ are (\emph{topologically}) \emph{equivalent} (this is symmetric).
Since every continuous map into a \HD\ space is determined by its restriction to any dense subset of its domain (cf. \cite[Corollary 13.14]{Willard}), the witness $f\colon (K',h')\ge (K,h)$ is unique provided that $K$ is \HD .

\begin{lemma}[{\cite[Lemmas 19.7 and 19.8]{Willard}}]
\label{Top:compactification:witnessBehaviour}
Let $(K,h)$ and $(K',h')$ be two \HDcomp s of a topological space $X$. 
\begin{enumerate}
\item If $f\colon (K',h')\ge (K,h)$ then $f\rest h'[X]$ is a homeomorphism between $h'[X]$ and $h[X]$, and $f[K'\setminus h'[X]]=K\setminus h[X]$.
\item The \comp s $(K,h)$ and $(K',h')$ are topologically equivalent if and only if both $(K,h)\le (K',h')$ and $(K',h')\le (K,h)$ hold.
\end{enumerate}
\end{lemma}

A \emph{one-point compactification} is a compactification with singleton remainder.
It is well known (cf.~\cite[Theorem 29.1]{Munkres}) that a topological space $X$ has a one-point \HD\ \comp\ $(\omega X,\iota)$ if and only if $X$ is locally compact\footnote{A topological space $X$ is \emph{locally compact} if for each of its points there is some compact subspace of $X$ which includes an open neighbourhood of that point.} and Hausdorff but not compact, and that $(\omega X,\iota)$ is unique up to topological equivalence.

Suppose now that $X$ is a discrete topological space.
Since $X$ is locally compact, $X$ is open in all of its \HDcomp s 
(cf.~\cite[Theorem~3.6.6]{EngelkingBook}).
\begin{itemize}
\item If $X$ is infinite and $\ast$ is a point that is not in $X$ we can extend $X$ to its one-point \HD\ compactification $\omega X:=X\sqcup\{\ast\}$ by declaring as open in addition to the open sets of $X$, for every finite $A\subseteq X$, the sets $\omega X\setminus A$ and taking the topology on $\omega X$ this generates.

\item If we pair the space $\beta X$ of all ultrafilters on $X$ carrying the topology whose basic open sets are of the form $\{U\in\beta X\mid A\in U\}$, one for each $A\subseteq X$, with the embedding that sends every $x\in X$ to the principal ultrafilter on $X$ generated by $\{x\}$, then this yields the finest \HDcomp\ of $X$, its Stone-Čech \comp\ (which is unique up to topological equivalence).
By the \SC\ property every continuous function $f\colon X\to T$ into a compact \HD\ space $T$ has a continuous extension $\beta f\colon\beta X\to T$ with $\beta f\rest X=f$ (cf.~\cite[Theorem~3.5.1]{EngelkingBook}).
\end{itemize}

\begin{theorem}[{\cite[Corollary 7.4]{TheoryOfUltrafilters}}]
\label{Top:compactification:StoneCechCard}
If $X$ is an infinite set, then $\vert\beta X\vert = 2^{2^{\vert X\vert}}$.
\end{theorem}

\subsection{Graphs with ends, and inverse limits}
\label{SummaryEndsAndTangles}
Given a graph $G=(V,E)$ we write $\cX$ for the collection of all finite subsets of its vertex set $V$, partially ordered by inclusion.
A (combinatorial) \emph{end} of a graph is an equivalence class of rays, where a \emph{ray} is a 1-way infinite path.
Two rays are \emph{equivalent} if for every $X\in\cX$ both have a subray (also called \emph{tail}) in the same component of $G-X$.
In particular, for every end $\omega$ of $G$ there is a unique component of $G-X$ in which every ray of $\omega$ has a tail, and we denote this component by $C(X,\omega)$.
The set of ends of a graph $G$ is denoted by $\Omega=\Omega(G)$. 
Further details on ends as well as any graph-theoretic notation not explained here can be found in Diestel's book~\cite{Bible}, especially in Chapter~8.

If $\omega$ is an end of $G$, then the components $C(X,\omega)$ are compatible in that they form a limit of an inverse system. Before we provide more details, we dedicate a paragraph to the definition of an inverse limit:

A partially ordered set $(I,\le)$ is said to be \emph{directed} if for every two $i,j\in I$ there is some $k\in I$ with $k\ge i,j$.
Let $(\,X_i\mid i\in I\,)$ be a family of topological spaces indexed by some directed poset $(I,\le)$.
Furthermore, suppose that we have a family $(\,\varphi_{ji}\colon X_j\to X_i\,)_{i\le j\in I}$ of continuous mappings which are the identity on $X_i$ in case of $i=j$ and which are \emph{compatible} in that $\varphi_{ki}=\varphi_{ji}\circ\varphi_{kj}$ for all $i\le j\le k$. Then both families together are said to form an \emph{inverse system}, and the maps $\varphi_{ji}$ are called its \emph{bonding maps}. 
We denote such an inverse system by $\{X_i,\varphi_{ji},I\}$ or $\{X_i,\varphi_{ji}\}$ for short if $I$ is clear from context.
Its \emph{inverse limit} $\invLim{} X_i=\invLim{}(\,X_i\mid i\in I\,)$ is the topological space
\begin{align*}
\invLim{} X_i=\{\,(x_i)_{i\in I}\mid \varphi_{ji}(x_j)=x_i\text{ for all }i\le j\,\}\subseteq \prod_{i\in I}X_i.
\end{align*}
Whenever we define an inverse system without specifying a topology for the spaces $X_i$ first, we tacitly assume them to carry the discrete topology.
If each $X_i$ is (non-empty) compact \HD , then so is $\invLim{}X_i$.

Now we describe an inverse system giving the end space:
We note that $\cX$ is directed by inclusion, and for every $X\in\cX$ we let $\cC_X$ be the set of components of $G-X$.
Then letting $\cb_{X',X}\colon \cC_{X'}\to\cC_X$ for $X'\supseteq X$ send each component of $G-X'$ to the unique component of $G-X$ including it turns the sets $\cC_X$ into an inverse system $\{\cC_X,\cb_{X',X},\cX\}$. Clearly, its inverse limit consists precisely of the directions of the graph: choice maps $f$ assigning to every $X\in\cX$ a component of $G-X$ such that $f(X')\subseteq f(X)$ whenever $X'\supseteq X$.
In 2010, Diestel and Kühn~\cite{Ends} showed that

\begin{theorem}[{\cite[Theorem 2.2]{Ends}}]\label{EndsAreDirections}
Let $G$ be any graph. Then there is a canonical bijection between the (combinatorial) ends of $G$ and its directions, i.e. $\Omega =\invLim{}\cC_X$.
\end{theorem}

Before we provide details on the Freudenthal \comp , we turn $G$ into a topological space.
In the \emph{1-complex} of $G$ which we denote also by $G$, every edge $e=xy$ is a homeomorphic copy $[x,y]:=\{x\}\sqcup\mathring{e}\sqcup\{y\}$ of $[0,1]$ with $\mathring{e}$ corresponding to $(0,1)$. 
The point set of $G$ is $V\sqcup\bigsqcup_{e\in E}\mathring{e}$.
Points in $\mathring{e}$ are called \emph{inner edge points}, and they inherit their basic open neighbourhoods from $(0,1)$. The space $[x,y]$ is called a \emph{topological edge}, but we refer to it simply as \emph{edge} and denote it by $e$.
For each subcollection $F\subseteq E$ we write $\mathring{F}$ for the set $\bigsqcup_{e\in F}\mathring{e}$ of inner edge points of edges in $F$.
The basic open neighbourhoods of a vertex $v$ of $G$ are given by unions $\bigcup_{e\in E(v)}[v,i_e)$ of half open intervals with each $i_e$ some inner edge point of $e$.
The 1-complex of $G$ is compact if and only if the graph $G$ is finite.

\begin{convention}
For edges $e$ and edge sets $F$ we always mean $\mathring{e}$ and $\mathring{F}$ in the sense above, not the interior with respect to some ambient topological space.
\end{convention}

We extend (the 1-complex of) $G$ to a topological space $\VG=G\sqcup\Omega$ by declaring as open in addition to the open sets of $G$, for all $X\in\cX$ and all $\cC\subseteq\cC_X$, the sets
\begin{align*}
\cO_{\VG}(X,\cC):=\medcup\cC\cup\mathring{E}(X,\medcup\cC)\cup\Omega(X,\cC)
\end{align*}
and taking the topology on $\VG$ that this generates.
Here, $\Omega(X,\cC)$ denotes the collection of those ends $\omega$ of $G$ with $C(X,\omega)\in\cC$.
Given $X\in\cX$ and an end $\omega$ of $G$ we write $\hat{C}(X,\omega)$ for $\cO_{\VG}(X,\{C(X,\omega)\})$.
For graphs $G$ that are locally finite and connected, their Freudenthal \comp\ coincides with $\VG$.
For arbitrary $G$ this is not true. However, $\VG$ still is a reasonable extension of $G$ also in the non-locally finite case, with a new point living at each end of the graph. But beware that $\VG$ is compact if and only if every $\cC_X$ is finite (cf.~\cite[Theorem 4.1]{VTopComp}; we provide a short proof in Lemma~\ref{VGcompact}).

\subsection{Tangles}
Next, we formally introduce tangles for a particular type of `separation system', referring the reader to~\cite{AbstractSepSys} for an overview of the full theory and its applications.
A (\emph{finite order}) \emph{separation} of a graph $G$ is a set $\{A,B\}$ with $A\cap B$ finite and $A\cup B=V$ such that $G$ has no edge between $A\setminus B$ and $B\setminus A$.
The ordered pairs $(A,B)$ and $(B,A)$ are then called the \emph{orientations} of the separation $\{A,B\}$, or (\emph{oriented}) \emph{separations}.
Informally we think of $A$ and $B$ as the \emph{small side} and the \emph{big side} of $(A,B)$, respectively. 
Furthermore, we think of the separation $(A,B)$ as \emph{pointing towards} its big side $B$ and \emph{away from} its small side $A$.
If $S$ is a collection of unoriented separations, then we write $\vS$ for the collection of their orientations. 
A subset $O$ of $\vS$ is an \emph{orientation} of $S$ if it contains precisely one of $(A,B)$ and $(B,A)$ for each separation $\{A,B\}$ in $S$.

We define a partial ordering $\le$ on $\vS$ by letting
\begin{align*}
(A,B)\le (C,D):\Leftrightarrow A\subseteq C\text{ and }B\supseteq D.
\end{align*}
Here, we informally think of the oriented separation $(A,B)$ as \emph{pointing towards} $\{C,D\}$ and its orientations, whereas we think of $(C,D)$ as \emph{pointing away from} $\{A,B\}$ and its orientations.
If $O$ is an orientation of $S$ and no two distinct separations $(B,A)$ and $(C,D)$ in $O$ satisfy $(A,B)<(C,D)$, i.e., no two distinct separations in $O$ point away from each other, then we call $O$ \emph{consistent}.

We call a set $\sigma\subseteq\vS$ of oriented separations a \emph{star} (\emph{in} $S$) if every two distinct separations $(A,B)$ and $(D,C)$ in $\sigma$ point towards each other, i.e. satisfy $(A,B)\le (C,D)$.
The \emph{interior} of a star $\sigma=\{\,(A_i,B_i)\mid i\in I\,\}$ is the intersection $\bigcap_{i\in I}B_i$ of all the big sides.
We say that an orientation $O$ of $S$ \emph{avoids} a subcollection $\cF\subseteq\vS$ if no subset of $O$ is contained in $\cF$.

\begin{definition}
Let $S$ be a collection of finite order separations of a graph $G$ and let $\cF$ be a collection of stars in $S$. An \emph{$\cF$-tangle} (\emph{of $S$}) is a consistent orientation of $S$ that avoids $\cF$.
\end{definition}

\subsection{Ends and Tangles}
We conclude this section by giving a summary of Diestel's paper~\cite{EndsAndTangles}.
From now on, let $G=(V,E)$ be a fixed infinite graph and let $S$ be the collection of all its finite order separations.
We write $\cT_{<\aleph_0}$ for the set of all finite stars in $S$ of finite interior, and we write $\cT$ for the set of all stars in $S$ of finite interior (so $\cT_{<\aleph_0}\subseteq\cT$). 
Instead of $\cT_{<\aleph_0}$-tangles (of $S$) we say $\aleph_0$-tangles (of $G$), and we write $\Theta$ for the collection of all $\aleph_0$-tangles.\footnote{Diestel~\cite{EndsAndTangles} showed that this definition is equivalent to Robertson and Seymour's~\cite{GMX}.}
Clearly, every $\cT$-tangle is an $\aleph_0$-tangle.
If $\omega$ is an end of $G$, then letting
\begin{align*}
\tau_\omega:=\{\,(A,B)\in\vS\mid C(A\cap B,\omega)\subseteq G[B\setminus A]\,\}
\end{align*}
defines a bijection $\omega\mapsto\tau_\omega$ between the ends of $G$ and the $\cT$-tangles. Therefore, we call these tangles the \emph{end tangles} of $G$.
By abuse of notation we write $\Omega$ for the collection of all end tangles of $G$, so we have $\Omega\subseteq\Theta$.

In order to understand the $\aleph_0$-tangles that are not ends, Diestel studied an inverse limit description of $\Theta$ which we introduce in a moment.
First, we note that every finite order separation $\{A,B\}$ corresponds to the bipartition $\{\cC,\cC'\}$ of the component space $\cC_X$ with $X=A\cap B$ and
\begin{align*}
\{A,B\}=\big\{\,V[\cC]\cup X\,,\,X\cup V[\cC']\,\big\}
\end{align*}
where $V[\cC]=\bigcup_{C\in\cC}V(C)$,
and this correspondence is bijective for fixed $X\in\cX$. 
For all $\cC\subseteq\cC_X$ we write 
\begin{align*}
s_{X\to\cC}=\big(\,V\setminus V[\cC]\,,\,X\cup V[\cC]\,\big)\quad\text{and}\quad s_{\cC\to X}=\big(\,V[\cC]\cup X\,,\,V\setminus V[\cC]\,\big)
\end{align*}
whereas we write $s_{X\to C}$ and $s_{C\to X}$ instead of $s_{X\to\{C\}}$ and $s_{\{C\}\to X}$, respectively.
Hence if $\tau$ is an $\aleph_0$-tangle of the graph, then for each $X\in\cX$ it also chooses one \emph{big side} from each bipartition $\{\cC,\cC'\}$ of $\cC_X$, namely the $\cK\in\{\cC,\cC'\}$ with $s_{X\to\cK}\in\tau$. 
Since it chooses theses sides consistently, it induces an ultrafilter $U(\tau,X)$ on $\cC_X$, one for every $X\in\cX$, which is given by
\begin{align*}
U(\tau,X)=\{\,\cC\subseteq\cC_X\mid s_{X\to\cC}\in\tau\,\},
\end{align*}
and these ultrafilters are compatible in that they form a limit of the inverse system $\{\,\beta(\cC_X)\,,\,\beta(\cb_{X',X})\,,\,\cX\,\}$.
Here, each set $\cC_X$ is endowed with the discrete topology and $\beta(\cC_X)$ denotes its \SC\ \comp .
Every bonding map $\beta(\cb_{X',X})$ is the unique continuous extension of $\cb_{X',X}$ that is provided by the \SC\ property.
More explicitly, the map $\beta(\cb_{X',X})$ sends each ultrafilter $U'\in\beta (\cC_{X'})$ to its \emph{restriction}
\begin{align*}
U'\rest X=\{\,\cC\subseteq\cC_X\mid\exists\,\cC'\in U'\colon\cC\supseteq\cC'\rest X\,\}\in\beta (\cC_X)
\end{align*}
where $\cC'\rest X=\cb_{X',X}[\cC']$.
Resuming Diestel's notation, we write $\cU_X$ for $\beta(\cC_X)$ and $f_{X',X}$ for $\beta(\cb_{X',X})$.
As one of his main results, Diestel showed that the map
\begin{align*}
\tau\mapsto (\,U(\tau,X)\mid X\in\cX\,)
\end{align*}
defines a bijection between the tangle space $\Theta$ and the inverse limit $\cU:=\invLim\cU_X$.
Moreover, he showed that the ends of $G$ are precisely those $\aleph_0$-tangles whose induced ultrafilters are all principal.

For every $\aleph_0$-tangle $\tau$ we write $\cX_\tau$ for the collection of all $X\in\cX$ for which the induced ultrafilter $U(\tau,X)$ is free.
Equivalently, $\cX_\tau$ is the collection of those $X\in\cX$ for which the star $\{\,s_{C\to X}\mid C\in \cC_X\,\}$ is included in $\tau$.
The set $\cX_\tau$ is empty if and only if $\tau$ is an end tangle. 
An $\aleph_0$-tangle $\tau$ with $\cX_\tau$ non-empty is called an \emph{ultrafilter tangle}, and we write $\Upsilon$ for the collection of all ultrafilter tangles, i.e. $\Upsilon=\Theta\setminus\Omega=\cU\setminus\Omega$.
For every ultrafilter tangle $\tau$ the set $\cX_\tau$ has a least element $X_\tau$ of which it is the up-closure.
Later, we will characterise these elements combinatorially as the \emph{critical vertex sets} (cf.~Theorem~\ref{critX=critUF}).

\begin{theorem}[{\cite[Theorem 2 and Lemma 3.1]{EndsAndTangles}}]
\label{UF:everyUFonCXextendsToUF}
Given $X\in\cX$, each free ultrafilter $U$ on $\cC_X$ determines an ultrafilter tangle $\tau$ of $G$ with $U(\tau,X)=U$.
\end{theorem}

We conclude our summary of `Ends and tangles' with the formal construction of the tangle \comp .
To obtain the tangle \comp\ $\modG$ of a graph $G$ we extend the 1-complex of $G$ to a topological space $G\sqcup\Theta=G\sqcup\cU$ by declaring as open in addition to the open sets of $G$, for all $X\in\cX$ and all $\cC\subseteq\cC_X$, the sets
\begin{align*}
\cO_{\modG}(X,\cC):=\medcup\cC\cup\mathring{E}(X,\medcup\cC)\cup\big\{\,(\,U_Y : Y\in\cX\,)\in\cU\;\big\vert\; \cC\in U_X\,\big\}
\end{align*}
and taking the topology this generates.
\begin{theorem}[{\cite[Theorem 1]{EndsAndTangles}}]
\label{DiestelsTangleCompWorks}
Let $G$ be any graph.
\begin{enumerate}
\item $\modG$ is a \comp\ of $G$ with totally disconnected remainder.
\item If $G$ is locally finite and connected, then $\modG=\VG$ coincides with the Freudenthal \comp\ of $G$.
\end{enumerate}
\end{theorem}

Teegen~\cite{MTeegen} generalised the tangle \comp\ to topological spaces.

\section{Coarsest compactifications extending the end space}
\label{sec3}

When Diestel asked for the coarsest \comp\ of a graph ``in which its ends appear as distinct points'', he kept things informal deliberately.
Choosing a way to make this precise is where we start:

\begin{definition}
If $G$ is any graph, then we call a \comp\ $\alpha G$ of $\VG$ an \emph{\Ocomp} (\emph{of} $G$) if $\alpha G\setminus\mathring{E}$ is a \HDcomp\ of $\VG\setminus\mathring{E}$.
\end{definition}
\begin{example}
Diestel's tangle \comp\ $\modG$ is an \Ocomp\ of $G$.
If $G$ is locally finite and connected, then its Freudenthal \comp\ coincides with $\VG=\modG$ and hence is an \Ocomp\ of $G$.
\end{example}

Every \Ocomp\ of $G$ is in particular a \comp\ of $G$.
Requiring an \Ocomp\ to be a \comp\ of $\VG$ ensures that it extends the end space as well as the graph and endows $G\sqcup\Omega$ with a meaningful topology.
Considering only \HDcomp s of $G$ is not an option since the tangle \comp\ is not \HD\ (however, its singleton subsets are closed in it).
As a result in an upcoming paper~\cite{StoneCechTangles} we show that the remainder of a \HDcomp\ of any non-locally finite $G$ cannot be totally disconnected.
Since the tangle space $\Theta$ is totally disconnected, this means that there is no way to extend the topology of $\Theta$ to one on $G\sqcup\Theta$ so as to yield a \HDcomp\ of $G$.
In this sense the topology of the tangle \comp\ is best possible.

However, working with non-\HD\ \comp s can be cumbersome, and since $\modG\setminus\mathring{E}$ is \HD\ one might think that requiring in the definition of an \Ocomp\ only the subspace $\alpha G\setminus\mathring{E}$ to be \HD\ would allow us to apply standard results about \HDcomp s also to \Ocomp s.
But this requirement does not suffice to ensure that $\alpha G\setminus\mathring{E}$ is a \HDcomp\ of $\VG\setminus\mathring{E}$: indeed, $\mathring{E}$ need not be open in $\alpha G$ (recall that the notion $\mathring{E}$ does not depend on any topology) so $\alpha G\setminus\mathring{E}$ need not be compact, and moreover $\VG\setminus\mathring{E}$ need not be dense in $\alpha G\setminus\mathring{E}$ (e.g. some point in $\alpha G\setminus\VG$ might have an open neighbourhood basis in $\alpha G$ of sets meeting $\VG$ only in $\mathring{E}$).
That is why we decided to require $\alpha G\setminus\mathring{E}$ to be a \HDcomp\ of $\VG\setminus\mathring{E}$.

\begin{convention}
Even though we speak of an \Ocomp\ `of $G$', we formally treat it as a \comp\ of $\VG$. For example, if $\delta G\ge\alpha G$ holds for another \Ocomp\ $\delta G$, then any $f\colon\delta G\ge\alpha G$ is required to fix $\Omega$ as well as $G$.
Likewise, a one-point \Ocomp\ $\alpha G$ of $G$ is one with $\vert\alpha G\setminus\VG\vert=1$.
\end{convention}

As our first main result, we give a combinatorial characterisation of the graphs admitting a one-point \Ocomp . This requires some preparation.

\begin{definition}
We call a graph $G$ \emph{\sol} if deleting a finite set of vertices leaves only finitely many components.
An end $\omega$ of $G$ living in a \sol\ component $C(X,\omega)$ for some $X\in\cX$ we call \emph{\sol}.
A graph with all ends \sol\ we call \emph{\esol}.
\end{definition}

Tough graphs are \esol .
Readers familiar with the notion of $t$-tough will note the similarity to our definition of tough.
However, a \sol\ graph need not be $t$-tough for any $t$, so we decided to leave it at that.
It is known (cf.~\cite[Theorem 4.1]{VTopComp}) that precisely the \sol\ graphs are compactified by their ends.
Our next lemma derives this result from the compactness of the tangle \comp :
\begin{lemma}\label{VGcompact}
Let $G$ be any graph. The following are equivalent:
\begin{enumerate}
\item $\VG$ is compact.
\item $\VG\setminus\mathring{E}$ is compact.
\item $G$ is \sol .
\item $\Omega=\Theta$.
\end{enumerate}
\end{lemma}
\begin{proof}
(iv)$\to$(i)$\to$(ii). If all $\aleph_0$-tangles are ends, then $\modG$ coincides with $\VG$.

(ii)$\to$(iii). 
If (iii) fails there is some $X\in\cX$ with $\cC_X$ infinite, and then 
\begin{align*}
\big\{\,\{x\}\mid x\in X\,\big\}\cup\big\{\,\cO_{\VG}(X,\{C\})\setminus\mathring{E}\;\big\vert\; C\in\cC_X\,\big\}
\end{align*}
is an open cover of $\VG\setminus\mathring{E}$ which admits no finite subcover.

(iii)$\to$(iv).
If $G$ is \sol\ then $\cX_\tau$ is empty for every $\aleph_0$-tangle $\tau$.
\end{proof}

\begin{figure}
\centering
\includegraphics[scale=0.8]{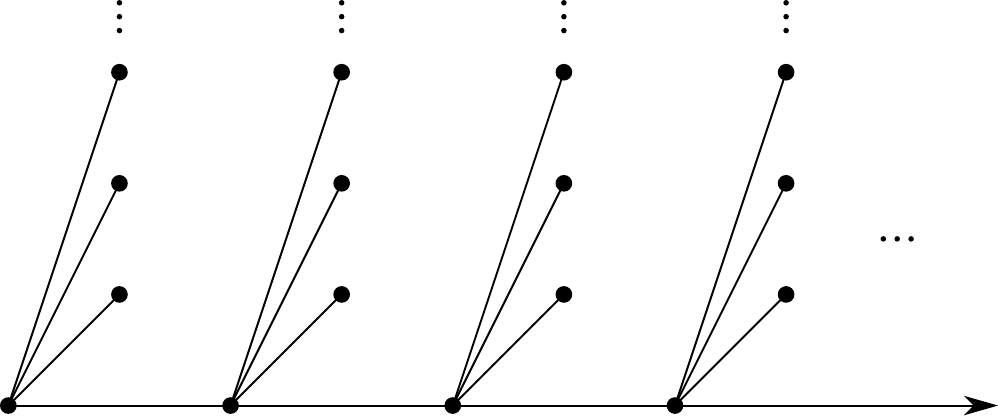}
\caption{This graph is neither \sol\ nor \esol\ and it has no \Ocomp\ with remainder of size at most one}
\label{fig:astRay}
\end{figure}

Equipped with this lemma, we are ready to investigate an example:

\begin{example}
Let $G$ be the \nsol\ graph depicted in Fig.~\ref{fig:astRay} and note that its sole end is not \sol .
Lemma~\ref{VGcompact} tells us that $\VG\setminus\mathring{E}$ is not compact.
Let us suppose that $G$ has a one-point \Ocomp\ $\alpha G$.
Then $\alpha G\setminus\mathring{E}$ is the one-point \HDcomp\ of $\VG\setminus\mathring{E}$, so $\VG\setminus\mathring{E}$ is locally compact and the sole end $\omega$ of $G$ has a compact neighbourhood $A$ in $\VG\setminus\mathring{E}$.
Consider some open neighbourhood $\hat{C}(X,\omega)\setminus\mathring{E}\subseteq A$ of $\omega$ in $\alpha G\setminus\mathring{E}$.
Then this open neighbourhood actually is a homeomorphic copy of the non-compact space $\VG\setminus\mathring{E}$, so we find a bad covering of it by open sets of $\VG\setminus\mathring{E}$.
Since $\omega$ is the only end of $G$ and $\{v\}$ is open in $\VG\setminus\mathring{E}$ for each vertex $v$ of $G$, we can extend the bad covering of that open neighbourhood to one of $A$ in $\VG\setminus\mathring{E}$ by adding singletons.
Therefore, $G$ is a graph that neither has a trivial \Ocomp\ nor a one-point \Ocomp .
\end{example}

Our next theorem states that the overall structure of this example is essentially the only obstruction to the existence of a one-point \Ocomp :

\begin{theorem}\label{1ptOcomp}
Let $G$ be any graph. The following are equivalent:
\begin{enumerate}
\item $G$ has a one-point $\Omega$-compactification.
\item $G$ is \esol\ but not \sol .
\end{enumerate}
In particular, every rayless graph has a one-point \Ocomp .
\end{theorem}
\begin{proof}
We start with the `in particular' part, assuming (ii)$\to$(i). 
Since $G$ has no end, it is \esol .
Hence by (ii)$\to$(i) it suffices to show that $G$ is not \sol .
And $G$ is indeed not \sol , since otherwise $\Omega=\invLim{}\cC_X$ is non-empty as inverse limit of non-empty finite spaces, yielding a ray in $G$ contrary to our assumptions.

(i)$\to$(ii). If $\alpha G$ is a one-point \Ocomp\ of $G$, then $\alpha G\setminus\mathring{E}$ is the one-point \HD\ \comp\ of $\VG\setminus\mathring{E}$, so $\VG\setminus\mathring{E}$ is locally compact but not compact. Hence $G$ is not \sol\ by Lemma~\ref{VGcompact} and it remains to verify \esoly .

For this, we assume for a contradiction that some end $\omega$ of $G$ is not \sol .
Since $\VG\setminus\mathring{E}$ is locally compact, we find a compact neighbourhood $A$ of $\omega$ in $\VG\setminus\mathring{E}$.
Next, we pick an open neighbourhood $\hat{C}(X,\omega)\setminus\mathring{E}\subseteq A$ of $\omega$ in $\VG\setminus\mathring{E}$.
Let $X'\in\cX$ witness the \nsoly\ of $C(X,\omega)$ and put $\Xi=X\cup X'$.
Then
\begin{align*}
\big\{\,\{x\}\mid x\in \Xi\,\big\}\cup\big\{\,\cO_{\VG}(\Xi,\{C\})\setminus\mathring{E}\;\big\vert\; C\in\cC_{\Xi}\,\big\}
\end{align*}
is a cover of $\VG\setminus\mathring{E}$ by open sets which admits no finite subcover of $A$, contradicting the compactness of $A$.

(ii)$\to$(i). 
We extend $\VG$ to a topological space $\alpha G=\VG\sqcup\{\ast\}$ by declaring as open in addition to the open sets of $\VG$, for every $X\in\cX$ and each cofinite subset $\cC$ of $\cC_X$ that contains all the \nsol\ components, the sets
\begin{align*}
\cO^\ast (X,\cC):=\medcup\cC\cup\mathring{E}(X,\medcup\cC)\cup\Omega(X,\cC)\cup\{\ast\}
\end{align*}
and taking the topology on $\alpha G$ this generates.
To see that this really generates a topology, it suffices to show that for every two neighbourhoods $\cO^\ast (X,\cC)$ and $\cO^\ast (Y,\cD)$ of $\ast$ there is some such neighbourhood of $\ast$ included in their intersection.
Since $\cC_X\setminus\cC$ is a finite set of \sol\ components of $G-X$, the set $\cC':=\cb_{X\cup Y,X}^{-1}(\cC)$ is a cofinite subset of $\cC_{X\cup Y}$ containing all the \nsol\ components of $G-(X\cup Y)$.
Similarly, $\cD':=\cb_{X\cup Y,Y}^{-1}(\cD)$ is cofinite in $\cC_{X\cup Y}$ and contains all the \nsol\ components of $G-(X\cup Y)$, yielding
\begin{align*}
\ast\in\cO^\ast(X\cup Y,\cC'\cap\cD')\subseteq\cO^\ast(X,\cC)\cap\cO^\ast(Y,\cD).
\end{align*}

Next, we verify that $\VG$ is an \Ocomp\ of $G$.
Since $G$ is \nsol , all basic open neighbourhoods of $\ast$ meet the vertex set:
Indeed, consider any basic open $\cO^\ast(X,\cC)$. If $\cC_X$ is infinite, then so is $\cC$. Otherwise $\cC_X$ is finite; then some component of $G-X$ is \nsol , since otherwise $G$ itself is \sol\ contrary to our assumptions. In both cases $\cC$ is non-empty, so $\cO^\ast(X,\cC)$ meets $V$ as claimed.
Consequently, $\VG$ is dense in $\alpha G$ and $\VG\setminus\mathring{E}$ is dense in $\alpha G\setminus\mathring{E}$.
To see that $\alpha G\setminus\mathring{E}$ is Hausdorff it suffices to find disjoint open neighbourhoods of an arbitrary end of $G$ and $\ast$ in $\alpha G$.
Given any end $\omega$ of $G$ we pick an $X\in\cX$ such that $C(X,\omega)$ is \sol . 
Then $\hat{C}(X,\omega)$ and $\cO^\ast(X,\cC_X\setminus \{C(X,\omega)\})$ are disjoint open neighbourhoods of $\omega$ and $\ast$ in $\alpha G$, respectively, as desired.

It remains to show that $\alpha G$ is compact.
For this, let $O=\cO^\ast(X,\cC)$ be any basic open neighbourhood of $\ast$ .
It suffices to show that $\alpha G\setminus O$ is compact.
Write $H$ for the subgraph $G-\bigcup\cC$.
Clearly, $\vert H\vert_\Omega$ is homeomorphic to $\alpha G\setminus O$, so it suffices to show that $\vert H\vert_\Omega$ is compact.
Since $\cC_X\setminus\cC$ is a finite set of \sol\ components of $G-X$, the graph $H$ is \sol , and hence $\vert H\vert_\Omega$ is compact by Lemma~\ref{VGcompact}.
\end{proof}

It is well-known that every continuous surjection $f\colon X \twoheadrightarrow Y$ from a compact space $X$ onto a Hausdorff space $Y$ gives rise to a homeomorphism between $Y$ and the quotient $X  / \{\,f^{-1}(y)\mid y \in Y\,\}$ over the fibres of $f$.
Thus, each \HDcomp\ is a quotient of all finer ones.
Since \Ocomp s may be non-\HD , proving a similar result for them takes some effort even if we consider only ones whose topology comes with a nice basis:

\begin{definition}
\label{def:crude}
In the context of a given graph $G$ we call a set $M$ \emph{crude} if it satisfies $M\cap\mathring{E}=\bigcup_{v\in M\cap V}\mathring{E}(v)$.
If a topological space $X\supseteq G$ has a basis consisting of the basic open sets of $G$ and sets that are crude,
then we call both the basis and the space $X$ \emph{crude}.
\end{definition}

\begin{example}
Both $\VG$ and $\modG$ are crude, and so is the one-point \Ocomp\ constructed in the proof of Theorem~\ref{1ptOcomp}.
\end{example}

\begin{lemma}\label{MagicLemma}
If $X\supseteq G$ is a crude topological space, then every open set $O$ of $X\setminus\mathring{E}$ extends to an open set $O\cup\bigcup_{v\in O\cap V}\mathring{E}(v)$ of $X$.
\end{lemma}
\begin{proof}
We let $\cB$ be a crude basis for the topology of $X$ and note that $G$ is open in $X$. 
Since $O$ is open in $X\setminus\mathring{E}$ there is some open set $W$ of $X$ with $W\setminus\mathring{E}=O$ and $W\supseteq \bigcup_{v\in O\cap V}\mathring{E}(v)$.
Next, we choose a subcollection $\cO\subseteq\cB$ of basic open sets with $W=\bigcup\cO$.
If $\hat{O}:=O\cup\bigcup_{v\in O\cap V}\mathring{E}(v)$ is a proper subset of $W$, then there is some $B\in\cO$ with $B\not\subseteq\hat{O}$.
This $B$ cannot be crude, and hence must be basic open in $G$.
Since $B$ cannot be a basic open neighbourhood of a vertex of $G$, there is an edge $e$ of $G$ with $B\subseteq\mathring{e}$ and no endvertex of this edge may lie in $O$.
In particular, $B$ avoids $\hat{O}$.
Therefore, discarding every such $B$ from $\cO$ results in $\hat{O}=\bigcup\cO$.
\end{proof}

\begin{proposition}\label{MagicProp}
Let $G$ be any graph, and let $\alpha G$ and $\delta G$ be two crude \Ocomp s of $G$ with $f\colon\delta G\ge\alpha G$.
Then $\alpha G$ and $\delta G / \{\,f^{-1}(\xi)\mid\xi\in\alpha G\,\}$ are topologically equivalent.
\end{proposition}
\begin{proof}
Let ${\sim}_f$ be the equivalence relation on the remainder of $\delta G$ induced by $f$, namely the one with $\delta G/{\sim}_f=\delta G / \{\,f^{-1}(\xi)\mid\xi\in\alpha G\,\}$.
The map $f$ also gives rise to a continuous bijection $F\colon \delta G/{\sim}_f\to\alpha G$.
Since $(\delta G/{\sim}_f)\setminus\mathring{E}$ is compact and $\alpha G\setminus\mathring{E}$ is \HD , general topology yields that the restriction $\check{F}$ of $F$ to $(\delta G/{\sim}_f)\setminus\mathring{E}$ is a homeomorphism onto $\alpha G\setminus\mathring{E}$.
We use this to show that $F$ is open.

For this, let any open set $O$ of $\delta G/{\sim}_f$ be given, and let $\cB$ be a crude basis for the topology of $\delta G$.
Since $\bigcup O$ is open in $\delta G$, we find a subcollection $\cO\subseteq\cB$ with $\bigcup O=\bigcup\cO$.
Since ${\sim}_f$ only affects the remainder of $\delta G$ and $F$ is the identity on $G$ we may assume without loss of generality that every set in $\cO$ is crude. 
Then $\bigcup O$ is crude as union of crude sets, so $O$ is crude as well.
By Lemma~\ref{MagicLemma} the open set $\check{W}:=\check{F}[O\setminus\mathring{E}]$ of $\alpha G\setminus\mathring{E}$ extends to the open set $W:=\check{W}\cup\bigcup_{v\in\check{W}\cap V}\mathring{E}(v)$ of $\alpha G$.
Clearly, $W$ and $F[O]$ agree on $\alpha G\setminus\mathring{E}$.
The fact that $O$ is crude implies that $W$ and $F[O]$ agree also on $\mathring{E}$.
\end{proof}

The following lemma is folklore; we include it for the sake of convenience:
\begin{lemma}\label{HDcompQuotient}
Let $X$ be a topological space with a \HDcomp\ $\alpha X$, and let ${\approx}$ be a non-trivial equivalence relation on the remainder such that $\alpha X/{\approx}$ is again a \HDcomp\ of $X$.
Then $\alpha X /{\approx}\lneq \alpha X$.
\end{lemma}
\begin{proof}
The non-injective quotient map is the unique witness of $\alpha X\ge\alpha X/{\approx}$, thus $\alpha X\le\alpha X/{\approx}$ would force it to be a homeomorphism (cf.~Lemma~\ref{Top:compactification:witnessBehaviour}).
\end{proof}
As a consequence of this lemma, every graph with a vertex of infinite degree has no coarsest \HDcomp , but we will not use this observation.
We are now ready to answer Diestel's questions:

\begin{theorem}
\label{thm_Q1}
Let $G$ be any graph. 
The following cases can occur:
\begin{enumerate}
\item 
If $G$ is \sol , then every \Ocomp\ of $G$ coincides with $\VG$.
\item If $G$ is not \sol\ but \esol , then $G$ has a crude one-point \Ocomp\ which is the coarsest crude \Ocomp ; which is a quotient of every crude \Ocomp ; and which is not equivalent to $\modG$.
\item Lastly, if $G$ is neither \sol\ nor \esol , then $G$ has no \Ocomp\ with remainder of size at most one. Moreover, $G$ has no coarsest \Ocomp , not even if we consider only crude \Ocomp s.
\end{enumerate}
In particular, $\modG$ is the coarsest \Ocomp\ of $G$ if and only if $G$ is \sol .
\end{theorem}

\begin{proof}
(i). If $G$ is \sol , then $\VG\setminus\mathring{E}$ is compact by Lemma~\ref{VGcompact}.
In particular, $\VG\setminus\mathring{E}$ is compact \HD\ and hence every \HDcomp\ of $\VG\setminus\mathring{E}$ must have empty remainder.

(ii). 
Lemma~\ref{VGcompact} ensures that no \Ocomp\ of $G$ has empty remainder.
Let $\alpha G =\VG\sqcup\{\ast\}$ be the crude one-point \Ocomp\ of $G$ that we constructed in the proof of Theorem~\ref{1ptOcomp}.
Given any crude \Ocomp\ $\delta G$ of $G$ we have to find an $f\colon\delta G\ge\alpha G$.
Since $\alpha G\setminus\mathring{E}$ is the one-point \HDcomp\ of $\VG\setminus\mathring{E}$, we find some $\check{f}\colon\delta G\setminus\mathring{E}\ge\alpha G\setminus\mathring{E}$, and we put $f=\check{f}\cup\id_G$.
We must show that $f$ is continuous.
For this, let $O$ be any basic open set of $\alpha G$. 
If $O$ is open in $G$ we are done, so we may assume that $O$ is of the form $\cO^{\ast}(X,\cC)$.
In particular, $O$ is crude.
By Lemma~\ref{MagicLemma} the open set $\check{f}^{-1}(O\setminus\mathring{E})$ of $\delta G\setminus\mathring{E}$ extends to an open set of $\delta G$ which coincides with $f^{-1}(O)$ since $O$ is crude. 
Thus $f$ is continuous. By Proposition~\ref{MagicProp} we know that $\alpha G$ is a quotient of $\delta G$.
To see that $\alpha G$ and $\modG$ cannot be topologically equivalent note that $\alpha G$ has singleton remainder while $\modG\setminus\VG=\Upsilon$ has size at least two by Theorem~\ref{UF:everyUFonCXextendsToUF}.

(iii).
By Lemma~\ref{VGcompact} and Theorem~\ref{1ptOcomp} we have $\vert\alpha G\setminus\VG\vert\ge 2$.
Hence we may choose some distinct two points $x$ and $y$ in $\alpha G\setminus\VG$.
Then $\alpha G/\{x,y\}$ is again an \Ocomp\ of $G$, and $\alpha G\le\alpha G/\{x,y\}$ is impossible since otherwise 
\begin{align*}
\alpha G\setminus\mathring{E}\le (\alpha G/\{x,y\})\setminus\mathring{E}=(\alpha G\setminus\mathring{E})/\{x,y\}
\end{align*}
would contradict Lemma~\ref{HDcompQuotient}. If $\alpha G$ is crude, then so is $\alpha G/\{x,y\}$.
\end{proof}

This answers Diestel's questions, but the existence of a giant class of graphs that do not have a coarsest \Ocomp\ (crude or not) is not a satisfying answer if one hopes to find interesting \comp s that might help generalising results about locally finite graphs to arbitrary ones.
That is why we do not stop here.

\section{Compactifying any graph with ends and critical vertex sets}
\label{sec4}

In this section we introduce critical vertex sets and show how they can be used together with the ends to compactify an arbitrary graph.

We call a finite set $X$ of vertices of a graph \emph{critical} if deleting $X$ leaves some infinitely many components each with neighbourhood precisely equal to $X$.
More formally, we introduce some notation first:

\begin{notation}
For every $X\in\cX$ and each $Y\subseteq X$ we write $\cC_X(Y)$ for the collection of all components $C\in\cC_X$ with $N(C)=Y$.
\end{notation}

\begin{definition}
\label{criticalDef}
A finite set $X\in\cX$ is \emph{critical} if $\cC_X(X)$ is infinite.
\end{definition}
\begin{notation}
The collection of all critical elements of $\cX$ is denoted by $\crit(G)$.
Given $X\in\cX$ we write $\crit(X)$ for the collection $\crit(G)\cap 2^X$ of all critical subsets of $X$.
\end{notation}

The following two lemmas will be used all the time without further mentioning:

\begin{lemma}
\label{powersetInducesFinPart}
The power set of $X\in\cX$ induces a finite partition of $\cC_X$, namely
\[
\pushQED{\qed} 
\{\,\cC_X(Y)\mid Y\in 2^X\,\}\setminus\{\emptyset\}.\qedhere
\popQED
\]
\end{lemma}

\begin{lemma}\label{critXchar}
For every $X\in\cX$ we have
\[
\crit(X)=\{\,Y\in 2^X\mid\cC_X(Y)\text{ is infinite}\,\}.
\]
\end{lemma}
\begin{proof}
This is immediate from $\cC_X(Y)=\{\,C\in\cC_Y(Y)\mid C\cap X=\emptyset\,\}$.
\end{proof}

\begin{notation}\label{CXminus}
For all $X\in\cX$ we write $\cC_X^-$
for the finite set of those components of $G-X$ that are not contained in $\cC_X(Y)$ for any critical $Y\in\crit(X)$, i.e.
\begin{align*}
\cC_X^-=\cC_X\setminus\bigsqcup_{Y\in\crit(X)}\cC_X(Y).
\end{align*}
\end{notation}

If $Y\in\cX$ is critical, then there are infinitely many independent paths between any two distinct vertices in $Y$.
Therefore, if $X\in\cX$ does not include $Y$ there is a unique component of $G-X$ that meets $Y$.
Since every component in $\cC_Y(Y)$ sends an edge to every vertex in the non-empty set $Y\setminus X$, all of the components in $\cC_Y(Y)$ avoiding $X$ are included in the same component of $G-X$ as $Y\setminus X$.

\begin{notation}
For every $X\in\cX$ and $Y\in\crit(G)\setminus 2^X$ we write $C_X(Y)$ for the unique component of $G-X$ meeting $Y$ (equivalently: including $\bigcup\cC_{X\cup Y}(Y)$).
\end{notation}

Now that we are familiar with the basics of critical vertex sets, our next aim is to link them to Diestel's ultrafilter tangles:

\begin{lemma}\label{UF:criticalComponentsLiveInUF}
If $\tau$ is an ultrafilter tangle, then $\cC_{X_\tau}(X_\tau)\in U(\tau,X_\tau)$.
\end{lemma}
\begin{proof}
We recall that $U(\tau,X_\tau)$ is free and that $\cC_{X_\tau}$ admits the finite partition 
\begin{align*}
\{\,\cC_{X_\tau}(Y)\mid Y\in 2^{X_\tau}\,\}\setminus\{\emptyset\}.
\end{align*}
By Lemma~\ref{critXchar} there is some unique $Z\in\crit(X_\tau)$ with $\cC_{X_\tau}(Z)\in U(\tau,X_\tau)$. 
We assume for a contradiction that $Z$ is distinct from $X_\tau$ and write $\cC=\cC_{X_\tau}(Z)$.
In particular, $Z$ is a proper subset of $X_\tau$, so the ultrafilter $U(\tau,Z)$ is principal and hence generated by $\{C\}$ for some component $C$ of $G-Z$.
The components in $\cC$ are also components of $G-Z$, so we have $\cC=\cC\rest Z\in U(\tau,Z)$ and $C\in\cC$ follows.
By $U(\tau,Z)=U(\tau,X_\tau)\rest Z$ we find some $\cC'\in U(\tau,X_\tau)$ with $\cC'\rest Z\subseteq \{C\}$. 
Since $C\in\cC$ is a component of $G-X_\tau$ as well, the only possibility for $\cC'$ is $\{C\}$, so $\{C\}\in U(\tau,X_\tau)$ is the desired contradiction. 
\end{proof}

\begin{corollary}\label{XtauCrit}
If $\tau$ is an ultrafilter tangle, then $X_\tau$ is critical.\qed
\end{corollary}

\begin{corollary}\label{UF:criticalComponentsLiveInUFgeneralised}
If $\tau$ is an ultrafilter tangle and $X\in\cX_\tau$, then $\cC_X(X_\tau)\in U(\tau,X)$.
\end{corollary}
\begin{proof}
We write $\cC=\cC_{X_\tau}(X_\tau)$.
This set is contained in the free ultrafilter $U(\tau,X_\tau)$ by Lemma~\ref{UF:criticalComponentsLiveInUF}.
Let $\cD$ be the collection obtained from $\cC$ by discarding the finitely many components meeting $X$ from it, i.e. let $\cD=\cC\cap\cC_X$.
Then $\cD=\cC_X(X_\tau)$ holds, and $\cD$ being a cofinite subset of $\cC$ implies $\cD\in U(\tau,X_\tau)$. 
Due to $U(\tau, X_\tau)=U(\tau,X)\rest X_\tau$ we find some $\cD'\in U(\tau,X)$ with $\cD'\rest X_\tau\subseteq\cD$.
Then $\cD\subseteq\cC_X$ implies $\cD'\rest X_\tau=\cD'$, so $\cD'\in U(\tau,X)$ and $\cD'\subseteq\cD\subseteq\cC_X$ imply $\cD\in U(\tau,X)$.
\end{proof}

\begin{lemma}\label{UF:criticalXinfiniteCexistsUF}
For all $X\in\cX$, every $Y\in\crit(X)$ and each free ultrafilter $U$ on $\cC_X(Y)$ there is a unique ultrafilter tangle $\tau$ with $U(\tau,X)\cap 2^{\cC_X(Y)}=U$, and this ultrafilter tangle $\tau$ satisfies $X_\tau=Y$.
\end{lemma}
\begin{proof}
We let $U'$ be the ultrafilter on $\cC_X$ given by the up-closure of $U$.
This $U'$ determines an ultrafilter tangle $\tau$ by Theorem~\ref{UF:everyUFonCXextendsToUF}. 
In particular, we have $Y\in\cX_\tau$. For every $Y^-\subsetneq Y$ the set $\cC_X(Y)\rest Y^-$ is a singleton in $U(\tau,Y^-)$ witnessing $Y^-\notin\cX_{\tau}$, so $Y=X_\tau$ follows from $\cX_\tau=\lfloor X_\tau\rfloor_\cX$.
\end{proof}

\begin{lemma}\label{UF:CXY}
For every ultrafilter tangle $\tau$ and each $X\in\cX\setminus\cX_\tau$ 
we do have $X_\tau\subseteq X\cup C_X(X_\tau)$ and the ultrafilter $U(\tau,X)$ is generated by $\{C_X(X_\tau)\}$.
\end{lemma}
\begin{proof}
By Corollary~\ref{XtauCrit} we have $X_\tau\in\crit(G)$. In particular, $X_\tau\subseteq X\cup C_X(X_\tau)$. Put $X'=X\cup X_\tau$. 
Then Corollary~\ref{UF:criticalComponentsLiveInUFgeneralised} yields $\cC_{X'}(X_\tau)\in U(\tau,X')$. 
Finally, we note that $\{C_X(X_\tau)\}=\cC_{X'}(X_\tau)\rest X\in U(\tau,X)$.
\end{proof}

\begin{definition}
On the set $\Upsilon$ of ultrafilter tangles we define the equivalence relation~${\sim}$ by letting
\begin{align*}
\tau\sim\tau'\;:\Leftrightarrow{}\,X_\tau=X_{\tau'}\;.
\end{align*}
\end{definition}

\begin{theorem}\label{critX=critUF}
Let $G$ be any graph.
\begin{enumerate}
\item The map $[\tau]_{\sim}\mapsto X_\tau$ is a bijection between $\Upsilon /{\sim}$ and $\crit(G)$.
\item For every critical $X\in\crit(G)$ the map
\begin{align*}
\tau\mapsto U(\tau,X)\cap 2^{\cC_X(X)}
\end{align*}
is a bijection between the ultrafilter tangles $\tau$ with $X_\tau=X$ and the free ultrafilters on $\cC_X(X)$. 
Moreover, the number of ultrafilter tangles $\tau$ with $X_\tau=X$ is $2^{2^\kappa}$ where $\kappa=\vert\cC_X(X)\vert\ge\aleph_0$.
\item $\vert\crit(G)\vert\le\vert V\vert$.
\item $\vert\crit(G)\vert\cdot 2^{\frakc}\le\vert\Upsilon\vert$.
\end{enumerate}
\end{theorem}
\begin{proof}
(i). This map is well-defined by Corollary~\ref{XtauCrit}.
By definition of ${\sim}$ it is injective, and it is surjective by Lemma~\ref{UF:criticalXinfiniteCexistsUF}.

(ii). This map is well-defined by Lemma~\ref{UF:criticalComponentsLiveInUF} and bijective by Lemma~\ref{UF:criticalXinfiniteCexistsUF}.
The number of free ultrafilters on $\cC_X(X)$ is $2^{2^\kappa}$ by Theorem~\ref{Top:compactification:StoneCechCard}.

(iii) is immediate from $\crit(G)\subseteq\cX=[V]^{<\aleph_0}$.

(iv) follows from (i) combined with (ii).
\end{proof}

As a consequence of Theorem~\ref{critX=critUF}~(i) the quotient $\modG/{\sim}$ witnesses that

\begin{theorem}\label{GisCompactifiedByEndsCrit}
Every graph is compactified by its ends plus critical vertex sets.\qed
\end{theorem}

Later, Theorem~\ref{minGquotientModG} will yield a more explicit description of the topology of the quotient $\modG/{\sim}$, and it will show that $\modG/{\sim}$ is a crude \Ocomp .

We close this section with a short lemma which we do not need for the remainder of this paper, but which we deem worth a few lines:

\begin{lemma}
The vertices of infinite degree are precisely the vertices that either dominate an end or lie in a critical vertex set.
\end{lemma}
\begin{proof}
It suffices to show that every vertex of infinite degree that is not contained in any critical vertex set dominates some end.
For this, let any such vertex be given and write $\cY$ for the directed poset formed by the $X\in\cX$ containing it.
Next, for all $X\in\cY$ we let $\cD_X$ consist of the components of $G-X$ to which our vertex sends an edge.
Since our vertex is not contained in any critical vertex set, we deduce that each $\cD_X$ is finite,
and our vertex having infinite degree ensures that no $\cD_X$ is empty.
Hence $\{\,\cD_X\,,\,\cb_{X',X}\rest\cD_{X'}\,,\,\cY\,\}$ is an inverse system of non-empty finite spaces, so its limit is non-empty.
Using Theorem~\ref{EndsAreDirections} and the fact that $\cY$ is cofinal in $\cX$, we obtain an end that is dominated by the given vertex.
\end{proof}

\section{Ends and critical vertex sets as tangles}
\label{sec:EndsCritAsTangles}
\label{sec5}

Theorem~\ref{GisCompactifiedByEndsCrit} raises the question whether it is possible to find a subset $S'$ of the set $S$ of finite order separations and a collection $\cF$ of stars in $S'$ such that the $\cF$-tangles of $S'$ are precisely the ends plus the critical vertex sets of $G$, i.e., whether $\modG/{\sim}$ is again a tangle-type \comp .
As the main result of this section, we show that this is the case.

\begin{definition}
Let $X\in\cX$ be given.
We call a subcollection $\cC\subseteq\cC_X$ \emph{tame} if for no $Y\in\crit(X)$ both $\cC_X(Y)\cap\cC$ and $\cC_X(Y)\cap (\cC_X\setminus\cC)$ are infinite.
If $\{\cC,\cC'\}$ is a bipartition of $\cC_X$ with both $\cC$ and $\cC'$ tame, then we call it \emph{tame}.
Furthermore, we call the corresponding finite order separation and its orientations \emph{tame}.
\end{definition}

\begin{example}
Finite subsets of $\cC_X$ are tame, and for all $Y\in 2^X$ each cofinite subset of $\cC_X(Y)$ is tame.
\end{example}

\begin{notation}
We write $\Pt{X}$ for the set of all tame subsets of $\cC_X$, partially ordered by inclusion.
We write $\St$ for the set of all tame finite order separations, and we write $\Ft$ for the set of all finite stars in $\St$ of finite interior. Instead of $\Ft$-tangles of $\St$ we shall say $\aleph_0$-tangles of $\St$, and we write $\Tt$ for the set of all $\aleph_0$-tangles of $\St$.
\end{notation}

Our first aim in this section is to find an inverse limit description of $\Tt$.
For this, we will show that every $\aleph_0$-tangle of $\St$ induces, for every $X\in\cX$, a particular type of filter on the poset $\Pt{X}$. 
However, we need some technical lemmas first:

\begin{lemma}\label{Tangles:FilterIntersection}
Given $X\in\cX$, any tangle $\tau\in\Tt$ containing $s_{X\to\cC}$ and $s_{X\to\cD}$ for some two subsets $\cC$ and $\cD$ of $\cC_X$ also contains $s_{X\to\cC\cap\cD}$.
\end{lemma}
\begin{proof}[{Proof (adapted from {\cite[Lemma 1.2]{EndsAndTangles}})}]
Given $X\in\cX$, a tangle $\tau\in\Tt$ and subsets $\cC$ and $\cD$ of $\cC_X$ with $s_{X\to\cC}\in\tau$ and $s_{X\to\cD}\in\tau$ we note first that $s_{X\to\cC\cap\cD}$ is tame so $\tau$ contains one of $s_{X\to\cC\cap\cD}$ and $s_{\cC\cap\cD\to X}$.
Assume for a contradiction that $\tau$ contains $s_{\cC\cap\cD\to X}$.
Clearly, $s_{X\to\cK}$ is tame for $\cK=\cC\cup (\cC_X\setminus \cD)$.
By consistency, $s_{X\to\cK}\le s_{X\to\cC}\in\tau$ implies $s_{X\to\cK}\in\tau$.
But then the star $\{\,s_{\cC\cap\cD\to X}\,,\,s_{X\to\cK}\,,\,s_{X\to\cD}\,\}$
has interior $X$ and is included in $\tau$, so $\tau$ does not avoid $\Ft$, a contradiction.
\end{proof}

Every $\aleph_0$-tangle $\tau$ of $\St$ induces, for every $X\in\cX$, the filter
\begin{align*}
F(\tau,X):=\{\,\cC\in\Pt{X}\mid s_{X\to\cC}\in\tau\,\}
\end{align*}
on the poset $\Pt{X}$ as the next lemma shows:

\begin{lemma}\label{Tangles:Filter}
For every $\tau\in\Tt$ and $X\in\cX$ the set $F(\tau,X)$ is a filter on $(\Pt{X},{\subseteq})$.
\end{lemma}
\begin{proof}
The star $\{s_{X\to\emptyset}\}\in\Ft$ is avoided by $\tau$, ensuring $\emptyset\notin F(\tau,X)$ as well as $\cC_X\in F(\tau,X)$.
For any two $\cC,\cD\in F(\tau,X)$ we have $\cC\cap\cD\in F(\tau,X)$ by Lemma~\ref{Tangles:FilterIntersection}.
Finally, for any $\cC\in F(\tau,X)$ and $\cD\in\Pt{X}$ with $\cC\subseteq\cD$ we also have $\cD\in F(\tau,X)$ by consistency of $\tau$.
\end{proof}

\begin{proposition}\label{Tangles:F(T,X)}
For every $\tau\in\Tt$ and $X\in\cX$ exactly one of the following holds:
\begin{enumerate}
\item There is a component $C$ of $G-X$ such that $F(\tau,X)$ is the principal filter on the poset $\Pt{X}$ at $\{C\}$, i.e. such that $F(\tau,X)=\lfloor\{C\}\rfloor_{\Pt{X}}$.
\item There is some $Y\in\crit(X)$ such that $F(\tau,X)$ is the up-closure in $\Pt{X}$ of the cofinite filter on $\cC_X(Y)$.
\end{enumerate}
\end{proposition}
\begin{proof}
We have seen in Lemma~\ref{Tangles:Filter} that $F(\tau,X)$ is a filter.
If $\tau$ contains $s_{X\to C}$ for some component $C$ of $G-X$, then (i) is the only possibility for our filter $F(\tau,X)$, so we may assume that no $s_{X\to C}$ is in $\tau$.
We recall that $\cC_X^-$ is the set of all components $C\in\cC_X$ that are not in $\cC_X(Y)$ for a critical $Y\in\crit(X)$.
The set 
\begin{align*}
\{\,s_{C\to X}\mid C\in\cC_X^-\,\}\cup\{\,s_{\cC_X(Y)\to X}\mid Y\in\crit(X)\,\}
\end{align*}
is a star in $\Ft$, so $\tau$ avoids it. 
Due to our assumption there is some $Z\in\crit(X)$ with $s_{X\to\cC_X(Z)}\in\tau$ witnessing that $\tau$ avoids this star, and this $Z$ is unique by consistency. 
Next, we verify that $F(\tau,X)$ includes the cofinite filter on $\cC_X(Z)$.
For this, let any cofinite subset $\cC$ of $\cC_X(Z)$ be given. As before, the set
\begin{align*}
\begin{array}{clcl}
& \{\,s_{C\to X}\mid C\in\cC_X^-\,\} & \cup & \big\{\,s_{\cC_X(Y)\to X}\mid Y\in\crit(X)\setminus\{Z\}\,\big\}\\
\cup & \{\,s_{C\to X}\mid C\in\cC_X(Z)\setminus\cC\,\} & \cup & \{\,s_{\cC\to X}\,\}
\end{array}
\end{align*}
is a star in $\Ft$ which $\tau$ must avoid. 
Since $\tau$ contains none of the $s_{X\to C}$ and none of the $s_{X\to \cC_X(Y)}$ for $Y\neq Z$ by the uniqueness of $Z$, it must contain $s_{X\to\cC}$, yielding $\cC\in F(\tau,X)$ as desired.
Thus $F(\tau,X)$ includes the cofinite filter on $\cC_X(Z)$.
Since $F(\tau,X)$ is a filter, it also includes the up-closure in $\Pt{X}$ of said cofinite filter.
If $F(\tau,X)$ is a proper superset of this up-closure, then this is witnessed by some $\cC\in F(\tau,X)$ with $\cC_X(Z)\setminus\cC$ infinite.
The separation $s_{X\to \cC}$ is tame, so $\cC\cap\cC_X(Z)$ must be finite.
But then $F(\tau, X)$ contains both $\cC_X(Z)\setminus\cC$ and $\cC\cap \cC_X(Z)$, so it also contains the empty set which is impossible.
Hence (ii) holds.
\end{proof}

This proposition already hints to the possibility of a connection between the $\aleph_0$-tangles of $\St$ and the ends plus the critical vertex sets of $G$.
Now we construct our inverse system: We take, for every $X\in\cX$, the set
\begin{align*}
\finC{X}:=\cC_X\sqcup\crit(X)
\end{align*}
and take the bonding maps $\eb_{X',X}\colon\finC{X'}\to\finC{X}$
\phantomsection
\label{eb_X'Xdef} 
for $X'\supseteq X$ which coincide with $\cb_{X',X}$ on $\cC_{X'}$; which are the identity on $\crit(X')\cap\crit(X)$; and which send each $Y\in\crit(X')\setminus\crit(X)$ to the unique component $C_X(Y)$ of $G-X$ meeting $Y$.
This completes the construction of the inverse system $\{\finC{X},\eb_{X',X},\cX\}$ whose inverse limit we denote by
\begin{align*}
\F:=\invLim{} (\,\finC{X}\mid X\in\cX\,).
\end{align*}

\begin{notation}
For every $Y\in\crit(G)$ we write $\langle Y\rangle$ for the limit $(\,p_X\mid X\in\cX\,)$ in $\F$ defined by setting $p_X=Y$ for all $X\in\lfloor Y\rfloor_{\cX}$ and $p_X=C_X(Y)$ otherwise.
\end{notation}

\begin{obs}\label{C=ends+crit}
The limits in $\F$ that are not ends are precisely the critical vertex sets, i.e. $\F=\Omega\sqcup\{\,\langle X\rangle\mid X\in\crit(G)\,\}$.\qed
\end{obs}

In order to link the $\aleph_0$-tangles of $\St$ to the limits of this inverse system, we define maps $\toC{}{X}\colon\Tt\to\finC{X}$, one for each $X\in\cX$, 
by letting them send every $\tau\in\Tt$ to the unique $C\in\cC_X$ or $Y\in\crit(X)$ given by Proposition~\ref{Tangles:F(T,X)}.
Once we have shown two technical lemmas, we shall see that these maps are compatible with the bonding maps of the inverse system, so they combine to a bijection between $\Tt$ and the inverse limit $\F=\invLim\finC{X}$.

\begin{lemma}\label{Tangles:S'Closed}
Let $\{A,B\}\in \St$ and $\{C,D\}\in S$ be such that both $A\triangle C$ and $B\triangle D$ are finite. Then $\{C,D\}\in \St$.
\end{lemma}
\begin{proof}
We assume for a contradiction that the separation $\{C,D\}$ is not in $\St$, witnessed by some $Y\in\crit(C\cap D)$.
Let $\{\cC,\cD\}$ be the bipartition of $\cC_{C\cap D}(Y)$ with $V[\cC]\subseteq C\setminus D$ and $V[\cD]\subseteq D\setminus C$.
By choice of $Y$, both $\cC$ and $\cD$ are infinite.
Next, we put
\begin{align*}
\cC'=\{\,K\in\cC\mid K\cap A\cap B=\emptyset\,\}\text{ and }\cD'=\{\,K\in\cD\mid K\cap A\cap B=\emptyset\,\},
\end{align*}
and observe that both $\cC\setminus\cC'$ and $\cD\setminus\cD'$ are finite since so is $A\cap B$.
By choice of $\{C,D\}$ we know that all but finitely many of the components in $\cC'$ are included in $G[A\setminus B]$ and all but finitely many of the components in $\cD'$ are included in $G[B\setminus A]$.
We write $\cC''$ and $\cD''$ for the infinite collections of these components, respectively.
Finally, we check two cases and derive a contradiction for both:

First, if $Y\notin\crit(A\cap B)$ then the component $C_{A\cap B}(Y)$ of $G-(A\cap B)$ includes both $\bigcup\cC''$ and $\bigcup\cD''$.
But exactly one of $G[A\setminus B]$ and $G[B\setminus A]$ includes $C_{A\cap B}(Y)$, contradicting $\bigcup\cC''\subseteq G[A\setminus B]$ and $\bigcup\cD''\subseteq G[B\setminus A]$ as desired.

Second, if $Y\in\crit(A\cap B)$, then $Y$ together with $\cC''$ and $\cD''$ 
witness $\{A,B\}\notin \St$, contradicting our assumptions.
\end{proof}

\begin{lemma}\label{Tangles:S'ClosedUnderFinDif}
Given $\tau\in\Tt$ and $(A,B)\in\tau$, if $(A',B')\in\vS$ is such that both $A\triangle A'$ and $B\triangle B'$ are finite, then $(A',B')\in\tau$.
\end{lemma}
\begin{proof}[{Proof (adapted from {\cite[Lemma 1.10]{EndsAndTangles}})}]
By Lemma~\ref{Tangles:S'Closed} all of the three separations $\{A',B'\}$, $\{A\cup A',B'\}$ and $\{A,B\cup B'\}$ are in $\St$.
It suffices to show $(A\cup A',B')\in\tau$, since then $(A',B')\in\tau$ follows from $(A',B')\le (A\cup A',B')\in\tau$ and $\tau$ being consistent.
Due to $(A,B\cup B')\le (A,B)\in\tau$ the consistency of $\tau$ implies that $(A,B\cup B')\in\tau$. 
Now the set $\{(A,B\cup B'),(B',A\cup A')\}$ is a star in $\Ft$, so the only possibility for $\{A\cup A',B'\}$ to be oriented by $\tau$ is $(A\cup A',B')\in\tau$.
\end{proof}

\begin{lemma}\label{Tangles:diagCommutes}
The diagram
\begin{equation*}
\begin{tikzcd}[column sep=1.5em]
& \Tt\arrow[dr,"\toC{}{X'}"]\arrow[dl,"\toC{}{X}"']\\
\finC{X} && \finC{X'}\arrow[ll,"\eb_{X',X}"']
\end{tikzcd}
\end{equation*}
commutes for all $X\subseteq X'\in\cX$.
\end{lemma}
\begin{proof}
Given $\tau\in\Tt$ we put $\xi=\eb_{X',X}(\toC{\tau}{X'})$ and check two cases: 

For the first case we assume that $\xi$ is a component $C$ of $G-X$, and we put $\cC'=\cb_{X',X}^{-1}(C)$.
It suffices to show $s_{X'\to\cC'}\in\tau$, since then Lemma~\ref{Tangles:S'ClosedUnderFinDif} yields $s_{X\to C}\in\tau$ so $\{C\}\in F(\tau,X)$ implies $\toC{\tau}{X}=C$ as desired.
For this, we first note that Lemma~\ref{Tangles:S'Closed} and $s_{X\to C}\in\vSt$ ensure $s_{X'\to\cC'}\in\vSt$.
Next, we claim that $\cC'\in F(\tau, X')$ holds:
Indeed, if $\toC{\tau}{X'}$ is a component of $G-X'$ then by definition of $\eb_{X',X}$ it must be a component in $\cC'$ and $\cC'\in F(\tau,X')$ follows.
And otherwise $\toC{\tau}{X'}$ is some critical $Y\in\crit(X')$ with $C_X(Y)=C$, so $\bigcup\cC_{X'}(Y)\subseteq C_X(Y)$ implies $\cC_{X'}(Y)\subseteq\cC'$, also resulting in $\cC'\in F(\tau,X')$.
Thus we have $\cC'\in F(\tau,X')$ which is tantamount to $s_{X'\to\cC'}\in\tau$.

For the second case we assume that $\xi$ is a critical vertex set $Y\in\crit(X)$. Then $\toC{\tau}{X'}=Y$ follows. We assume for a contradiction that $Y$ is distinct from $\toC{\tau}{X}$.
By definition of $\toC{\tau}{X}$ we find some cofinite subset $\cC$ of $\cC_X(Y)$ with $\cC\notin F(\tau,X)$.
To yield a contradiction, it suffices to show $s_{X\to\cC}\in\tau$.
For this, set $\cC'=\cb_{X',X}^{-1}(\cC)$. 
Then $\toC{\tau}{X'}=Y$ yields $\cC'\in F(\tau,X')$ since $\cC'$ is a cofinite subset of $\cC_{X'}(Y)$.
In particular, we have $s_{X'\to\cC'}\in\tau$ which implies $s_{X\to\cC}\in\tau$ by Lemma~\ref{Tangles:S'ClosedUnderFinDif}.
\end{proof}

\begin{theorem}
\label{OmegaAndCritAreTangles}
Let $G$ be any graph. 
The $\aleph_0$-tangles of $\St$ are precisely the limits of the inverse system $\{\finC{X},\eb_{X',X},\cX\}$, which in turn are precisely the ends and critical vertex sets of $G$, i.e. $\Tt=\F=\Omega\sqcup\crit(G)$.
\end{theorem}
\begin{proof}
We already noted $\F=\Omega\sqcup\crit(G)$ in Observation~\ref{C=ends+crit}.
The map 
\begin{align*}
\tau\mapsto (\,\toC{\tau}{X}\mid X\in\cX\,)
\end{align*}
from $\Tt$ to $\F$ is well-defined by Lemma~\ref{Tangles:diagCommutes}, and it is injective by definition: If $\tau$ and $\tau'$ are distinct tangles in $\Tt$, then this is witnessed by some separation $\{A,B\}$ with $(A,B)\in\tau\smallsetminus\tau'$ and $(B,A)\in\tau'\smallsetminus\tau$, so $F(\tau,A\cap B)$ and $F(\tau',A\cap B)$ are also distinct, causing $\toC{\tau}{A\cap B}\neq\toC{\tau'}{A\cap B}$.
Hence it remains to verify surjectivity. 

For this, let any $\xi\in\F$ be given.
If $\xi$ is an end $\omega$ of $G$, then $\tau_\omega\cap\vSt$ (here, $\tau_\omega$ is the $\aleph_0$-tangle of $S\supseteq \St$ induced by $\omega$) gets mapped to $\xi$.
Otherwise $\xi$ is of the form $\langle Y\rangle$ by Observation~\ref{C=ends+crit}.
Theorem~\ref{critX=critUF} yields an ultrafilter tangle $\tau$ (an $\aleph_0$-tangle of $S\supseteq \St$ that is not an end) with $X_\tau=Y$. 
Due to $\St\subseteq S$ and $\Ft\subseteq\cT_{<\aleph_0}$ it is immediate that $\tau\cap\vSt$ is an $\aleph_0$-tangle of $\St$.
It remains to check that it gets mapped to $\langle Y\rangle$.
For every $X\in\cX\setminus \cX_\tau$ the ultrafilter $U(\tau,X)$ is generated by $\{C_X(X_\tau)\}$ according to Lemma~\ref{UF:CXY}, so $\toC{\tau\cap\vSt}{X}=C_X(X_\tau)$ follows.
For every $X\in\cX_\tau$ the ultrafilter $U(\tau,X)$ is free and contains $\cC_X(X_\tau)$ by Corollary~\ref{UF:criticalComponentsLiveInUFgeneralised}, so $\toC{\tau\cap\vSt}{X}=X_\tau$ follows.
Thus $\tau\cap\vSt\in\Tt$ gets mapped to $\langle Y\rangle$ as desired.
\end{proof}

\section{Compactifications induced by \texorpdfstring{$\cC$}{C}-systems}
\label{sec6}

From a topological point of view, the compactness of the tangle compactification ultimately is a consequence of the \SC\ property giving rise to the compact \HD\ extension $\invLim{}\beta (\cC_X)=\cU$ of $\invLim{}\cC_X=\Omega$ and the way the inverse limit topology of $\cU$ is extended to interact with $G$ in $G\sqcup\cU=\modG$.
In the spirit of our paper, this raises the question whether there exists a coarsest compactification of $G$ among those that are induced in this particular way by the limit of a \emph{\Csys}, an inverse system of \HDcomp s of the discrete component spaces $\cC_X$ with bonding maps that continuously extend the underlying maps $\cb_{X',X}$.

As our two main results of this section we show that every \Csys\ gives rise to an \Ocomp\ of $G$ in the way Diestel used his \Csys\ $\{\cU_X,f_{X',X}\}$ to compactify $G$ in his tangle compactification, and we show how \Csys s can be partially ordered in a natural way that extends to the \Ocomp s they induce.
We will put these insights to use in the next section in order to find the coarsest \Ocomp\ that is induced by a \Csys .

\begin{definition}
We call an inverse system $\{\,(\alpha(\cC_X),\alpha_X)\,,\,\fraka_{X',X}\,,\,\cX\,\}$ of \HDcomp s $(\alpha(\cC_X),\alpha_X)$ of the discrete spaces $\cC_X$ a $\cC$-\emph{system} (\emph{of} $G$) if 
\begin{align}\label{Eq:CsystemCommute}
\fraka_{X',X}\circ\alpha_{X'}=\alpha_X\circ\cb_{X',X}
\end{align}
holds for all $X\subseteq X'\in\cX$, i.e. if the diagram
\begin{equation*}
\begin{tikzcd}
\cC_X\arrow[d, "\alpha_X"', right hook->] & \cC_{X'}\arrow[d, "\alpha_{X'}", right hook->]\arrow[l, "\cb_{X',X}"']\\
\alpha(\cC_X) & \alpha(\cC_{X'})\arrow[l, "\fraka_{X',X}"]
\end{tikzcd}
\end{equation*}
commutes for all $X\subseteq X'\in\cX$. 
\end{definition}
\begin{notation}
We write $\cC^\alpha$ for the \Csys\ $\{\,(\alpha(\cC_X),\alpha_X)\,,\,\fraka_{X',X}\,\}$ and $\cI^\alpha$ for its inverse limit $\invLim{}\alpha (\cC_X)$.
\end{notation}

Since every continuous map into a \HD\ space is determined by its restriction to any dense subset of its domain (cf. \cite[Corollary 13.14]{Willard}), condition~(\ref{Eq:CsystemCommute}) ensures that the bonding maps $\fraka_{X',X}$ are unique.

\begin{example}
If $G$ is \sol , then $\{\cC_X,\cb_{X',X}\}$ is a \Csys\ giving the end space $\Omega=\invLim{}\cC_X$ (cf.~Theorem~\ref{EndsAreDirections}) that compactifies $G$ in $\VG$.
\end{example}

\begin{example}
Diestel's $\{\cU_X,f_{X',X}\}$ is a \Csys\ giving the tangle space $\Theta=\cU=\invLim{}\cU_X$ that compactifies $G$ in his tangle \comp\ $\modG$.
\end{example}

\begin{notation}
We write $\cC^\cU$ for the \Csys\ $\{\cU_X,f_{X',X}\}$.
\end{notation}

By Theorem~\ref{EndsAreDirections} we have $\Omega=\invLim\cC_X$, so condition~(\ref{Eq:CsystemCommute}) ensures that the mapping
\begin{equation*}
\begin{aligned}
\begin{array}{rclc}
\iota^\alpha\;: & \Omega & \hookrightarrow & \cI^\alpha\\ 
 & (\,C_X\mid X\in\cX\,) & \mapsto & (\,\alpha_X(C_X)\mid X\in\cX\,)
\end{array} 
\end{aligned}
\end{equation*}
is a well-defined injection. 
As our first main result of this section, we generalise Diestel's construction of the tangle \comp\ and show that every \Csys\ gives rise to an \Ocomp :

Given any \Csys\ $\cC^\alpha$ of $G$ we let the map $\pi_X^\alpha\colon\cI^\alpha\to\alpha(\cC_X)$ be the continuous restriction of the $X$th projection map $\text{pr}_X\colon\prod_{Y\in\cX}\alpha(\cC_Y)\to\alpha(\cC_X)$ to $\cI^\alpha$. 
Now we extend the 1-complex of $G$ to a topological space $\alpha G=G\sqcup\cI^\alpha$ by declaring as open in addition to the open sets of $G$, for all $X\in\cX$ and every open set $O$ of $\alpha(\cC_X)$, the sets
\begin{align*}
\cO_{\alpha G}(X,O):=\medcup\cC\cup\mathring{E}(X,\medcup\cC)\cup (\pi_X^\alpha)^{-1}(O)
\end{align*}
where $\cC=\alpha_X^{-1}(O)\subseteq\cC_X$, and taking the topology on $\alpha G$ this generates; since it is not clear that we really defined a basis here,
we formally verify this in Lemma~\ref{Csystem:alphaGtopWellDef}.
By a general result on inverse limits, the open sets $(\pi_X^\alpha)^{-1}(O)$ of $\cI^\alpha$ form a basis for the topology of $\cI^\alpha$ (cf.~\cite[Lemma 1.1.1]{Field}), so $\alpha G$ includes $\cI^\alpha$ as a subspace.
For ease of notation we write $\pi_X$ instead of $\pi_X^\alpha$ if the affiliation is clear.
Before we prove that $\alpha G$ really is an \Ocomp\ of $G$, we check three technical facts:

\begin{fact}\label{Csystem:piXalphaCommute}
The diagram
\begin{equation*}
\begin{tikzcd}[column sep=1.5em]
& \cI^\alpha\arrow[dr,"\pi_{X'}^\alpha"]\arrow[dl,"\pi_X^\alpha"']\\
\alpha(\cC_X) && \alpha(\cC_{X'})\arrow[ll,"\fraka_{X',X}"'] 
\end{tikzcd}
\end{equation*}
commutes for all $X\subseteq X'\in\cX$.
\end{fact}

\begin{lemma}\label{Csystem:openNbhdsBehaveWell}
For all $X\subseteq X'\in\cX$ and every open set $O$ of $\alpha(\cC_X)$ we have
\begin{align*}
\cO_{\alpha G}(X,O)\supseteq\cO_{\alpha G}(X',\fraka_{X',X}^{-1}(O)).
\end{align*}
\end{lemma}
\begin{proof}
We write $O'=\fraka_{X',X}^{-1}(O)$. Fact~\ref{Csystem:piXalphaCommute} yields $\pi_X^{-1}(O)=\pi_{X'}^{-1}(O')$ so it remains to verify that $\bigcup \alpha_X^{-1}(O)\supseteq\bigcup \alpha_{X'}^{-1}(O')$ which is easily calculated:
\[
\medcup\alpha_{X'}^{-1}(O') =\; \medcup\alpha_{X'}^{-1}(\fraka_{X',X}^{-1}(O)) \overset{(\ref{Eq:CsystemCommute})}{=}\medcup\cb_{X',X}^{-1}(\alpha_X^{-1}(O))
\subseteq\;\medcup\alpha_X^{-1}(O). \qedhere
\]
\end{proof}

\begin{lemma}\label{Csystem:alphaGtopWellDef}
The open sets of the 1-complex of $G$ together with the sets $\cO_{\alpha G}(X,O)$ form a basis for a topology on $\alpha G=G\sqcup\cI^\alpha$.
\end{lemma}
\begin{proof}
It suffices to show that for every $\xi\in\cI^\alpha$ and every two neighbourhoods $\cO_{\alpha G}(X,O)$ and $\cO_{\alpha G}(X',O')$ of $\xi$ there exists a third neighbourhood of this form included in the intersection $\cO_{\alpha G}(X,O)\cap \cO_{\alpha G}(X',O')$.
Write $\Xi=X\cup X'$. The set
\begin{align*}
\cO_{\alpha G}\big(\,\Xi\,,\;\fraka_{\Xi,X}^{-1}(O)\cap\fraka_{\Xi,X'}^{-1}(O')\,\big)
\end{align*}
is such a neighbourhood by Fact~\ref{Csystem:piXalphaCommute} and Lemma~\ref{Csystem:openNbhdsBehaveWell}.
\end{proof}

\begin{theorem}\label{Csystem:EveryCsystemInducesOmegaComp}
Let $G$ be any graph.
If $\cC^\alpha$ is a $\cC$-system of $G$, then $\alpha G$ is an $\Omega$-compactification of $G$.
\end{theorem}
\begin{proof} 
Lemma~\ref{Csystem:alphaGtopWellDef} ensures that $\alpha G$ is a topological space.

First, we show that $\alpha G$ is compact.
For this, we generalise Diestel's proof of his \cite[Theorem 1 (i)]{EndsAndTangles} in that we replace his Lemmas 2.3 and 3.7 by topological arguments.
Let $\cO$ be any cover of $\alpha G\setminus G$ by open sets $\cO_{\alpha G}(X,O)$ of $\alpha G$.
The inverse limit $\cI^\alpha$ of compact \HD\ spaces is again compact \HD , so $\alpha G\setminus G=\cI^\alpha$ is compact and the cover $\cO$ admits a finite subcover of the form
\begin{align*}
\cO'=\{\,\cO_{\alpha G}(X,O_X)\mid X\in\cX'\,\}
\end{align*}
(with $\cX'\subseteq\cX$ finite) that covers $\cI^\alpha$.
Our aim is to show that $G\setminus\bigcup\cO'$ is the 1-complex of a finite graph, since then $G\setminus\bigcup\cO'=\alpha G\setminus\bigcup\cO'$ will be compact as desired.
For this, we put $\Xi=\bigcup\cX'$, and for each $X\in\cX'$ we let $O_X':=\fraka_{\Xi,X}^{-1}(O_X)$. 
Recall that $\cO_{\alpha G}(\Xi,O_X')\subseteq\cO_{\alpha G}(X,O_X)$ holds by Lemma~\ref{Csystem:openNbhdsBehaveWell}.
The collection
\begin{align*}
\{\,\cO_{\alpha G}(\Xi,O_X')\mid X\in\cX'\,\}
\end{align*}
still covers $\cI^\alpha$ by Fact~\ref{Csystem:piXalphaCommute}. Now we consider the set
\begin{align*}
\cC:=\cC_\Xi\setminus\bigcup_{X\in\cX'}\alpha_\Xi^{-1}(O_X').
\end{align*}
If $\bigcup\cC$ is finite, then $G[\Xi\cup V[\cC]]\supseteq G\setminus\bigcup \cO'$ is compact and we are done.
Hence we may assume for a contradiction that $\bigcup\cC$ is infinite.
The set
\begin{align*}
A:=\alpha (\cC_\Xi)\setminus\bigcup_{X\in\cX'}O_X'
\end{align*}
is closed in $\alpha(\cC_\Xi)$ and satisfies $\alpha_\Xi^{-1}(A)=\cC$. For all $Y\in\lfloor\Xi\rfloor_\cX$ put $A_Y=\fraka_{Y,\Xi}^{-1}(A)$.
Every $A_Y$ is compact \HD\ as closed subset of $\alpha (\cC_Y)$.
Since $\bigcup\cC$ is infinite, it follows that every $\cb_{Y,\Xi}^{-1}(\cC)$ is non-empty. 
Combined with $\alpha_\Xi[\cC]\subseteq A$ and~(\ref{Eq:CsystemCommute}) this implies that every $A_Y$ is non-empty, witnessed by $\alpha_Y[\cb_{Y,\Xi}^{-1}(\cC)]\subseteq A_Y$.
Consequently we find a limit of the inverse system $\{\,A_Y\,,\,\fraka_{Y',Y}\rest A_{Y'}\,,\,\lfloor\Xi\rfloor_\cX\,\}$ and this limit determines a $\xi\in\cI^\alpha$ since $\lfloor\Xi\rfloor_\cX$ is cofinal in $\cX$.
In particular,
\begin{align*}
\xi\in \pi_\Xi^{-1}(A)=\cI^\alpha\setminus\medcup\cO'
\end{align*}
is a contradiction. Thus $\alpha G$ is compact.

Second, we show that $\alpha G$ induces the correct subspace topology on $G\sqcup\iota^\alpha [\Omega]$.
For this we assume without loss of generality that $\iota^\alpha$ is the identity on $\Omega$.
Each basic open set $\cO_{\alpha G}(X,O)$ of $\alpha G$ induces on $G\sqcup\Omega$ the open set $\cO_{\VG}(X,\alpha_X^{-1}(O))$.
Conversely, every basic open set $\cO_{\VG}(X,\cD)$ of $\VG$ is induced by the basic open set $\cO_{\alpha G}(X,\alpha_X[\cD])$ of $\alpha G$ (recall that $\alpha_X[\cC_X]$ is open in $\alpha (\cC_X)$).

Finally, we deduce that $\alpha G$ is an \Ocomp\ of $G$.
We have shown that $\alpha G$ is a compact space including $\VG$ as a subspace. From $\cI^\alpha$ being \HD\ and the choice of our basis for the topology of $\alpha G$ it is immediate that $\alpha G\setminus\mathring{E}$ is \HD .
Since $\mathring{E}$ is open in $\alpha G$, it follows that $\alpha G\setminus\mathring{E}$ is compact.
Therefore, it remains to show that $\VG$ is dense in $\alpha G$ and that $\VG\setminus\mathring{E}$ is dense in $\alpha G\setminus\mathring{E}$.
For this, it suffices to show that an arbitrary basic open set $\cO_{\alpha G}(X,O)$ with $O$ non-empty meets $V$. 
Since $\alpha_X[\cC_X]$ is dense in $\alpha(\cC_X)$ we know that $O$ meets $\alpha_X[\cC_X]$, so $\bigcup\alpha_X^{-1}(O)$ is a non-empty subgraph of $G$, and hence $\cO_{\alpha G}(X,O)$ meets $V$.
\end{proof}

\begin{definition}
We call an \Ocomp\ of $G$ a \emph{\Ccomp } of $G$ if it is induced by a \Csys\ of $G$.
\end{definition}

\begin{fact}
All \Ccomp s are crude (cf.~p.~\pageref{def:crude}).
\end{fact}

Our next definition provides a way to compare \Csys s:

\begin{definition}
If $\cC^\alpha=\{\,(\alpha(\cC_X),\alpha_X)\,,\,\fraka_{X',X}\,\}$ and $\cC^\delta=\{\,(\delta(\cC_X),\delta_X)\,,\,\frakd_{X',X}\,\}$ are two \Csys s of $G$, then we write $\cC^\alpha\le_{\cC}\cC^\delta$ if
for every $X\in\cX$ there is some $f_X\colon (\delta(\cC_X),\delta_X)\ge (\alpha(\cC_X),\alpha_X)$ and these maps are compatible in that
\begin{align}\label{Eq:CsystemsComparisionCommute}
f_X\circ\frakd_{X',X}=\fraka_{X',X}\circ f_{X'}
\end{align}
holds for all $X\subseteq X'\in\cX$ (a diagram follows below).
\end{definition}
Recall that the maps $f_X$ are unique.
Condition~(\ref{Eq:CsystemsComparisionCommute}) together with condition~(\ref{Eq:CsystemCommute}) ensures that the left-hand diagram
\begin{equation*}\label{Diagram:CsystemsComparisionCommute}
\begin{tikzcd}
\cC_X\arrow[dr, red, right hook->, "\delta_X"]\arrow[ddr, red, right hook->, "\alpha_X"', bend right] 
  &
  & 
  & \cC_{X'}\arrow[dl, red, left hook->, "\delta_{X'}"']\arrow[ddl, red, left hook->, "\alpha_{X'}", bend left]\arrow[lll, "\cb_{X',X}"'] & &  \Omega\arrow[dl, red, left hook->, "\iota^\delta"']\arrow[ddl, red, left hook->, "\iota^\alpha", bend left]\\
 & \delta (\cC_X)\arrow[d, blue, "f_X"', two heads] 
  & \delta (\cC_{X'})\arrow[l, "\frakd_{X',X}"']\arrow[d, blue, "f_{X'}", two heads] & &  \cI^\delta\arrow[d, blue, two heads, "\psi^{\delta\alpha}"']\\
 & \alpha (\cC_X) & \alpha (\cC_{X'})\arrow[l, "\fraka_{X',X}"] & &  \cI^\alpha
\end{tikzcd}
\end{equation*}
commutes so that our compatible continuous surjections $f_X$ (cf.~Lemma~\ref{Top:compactification:witnessBehaviour}) combine to a well-defined continuous surjection
\begin{align*}
\begin{array}{rclc}
\psi^{\delta\alpha}\;: & \cI^\delta & \twoheadrightarrow & \cI^\alpha \\ 
 & (\,p_X\mid X\in\cX\,) & \mapsto & (\,f_X(p_X)\mid X\in\cX\,)
\end{array} 
\end{align*}
from one inverse limit onto the other (cf.~\cite[Corollary 1.1.5]{Field}), and that $\psi^{\delta\alpha}$ fixes $\Omega$ in that $\psi^{\delta\alpha}\circ\iota^\delta=\iota^\alpha$ (see the right-hand diagram above).
When we are given two concrete \Csys s $\cC^\alpha$ and $\cC^\delta$ later, verifying $\cC^\alpha\le_{\cC}\cC^\delta$ will be easy:

\begin{lemma}
\label{Csystem:leForCompsExtendsToleC}
If $\cC^\alpha$ and $\cC^\delta$ are \Csys s with 
$f_X\colon (\delta(\cC_X),\delta_X)\ge (\alpha(\cC_X),\alpha_X)$ for all $X\in\cX$, then~(\ref{Eq:CsystemsComparisionCommute}) holds for all $X\subseteq X'\in\cX$.
In particular, $\cC^\alpha\le_{\cC}\cC^\delta$.
\end{lemma}
\begin{proof}
We recall that the $f_X$ satisfy $f_X\circ\delta_X=\alpha_X$, and for $X'\supseteq X$ we compute
\begin{align*}
\begin{array}{ccccc}
f_X\circ\frakd_{X',X}\circ\delta_{X'} & \overset{(\ref{Eq:CsystemCommute})}{=} & f_X\circ\delta_X\circ\cb_{X',X} & = & \alpha_X\circ\cb_{X',X}\\
& \overset{(\ref{Eq:CsystemCommute})}{=} & \fraka_{X',X}\circ\alpha_{X'} & = & \fraka_{X',X}\circ f_{X'}\circ\delta_{X'}
\end{array}
\end{align*}
so both sides of~(\ref{Eq:CsystemsComparisionCommute}) agree on $\delta_{X'}[\cC_{X'}]$. 
Since every continuous map into a \HD\ space is determined by its restriction to any dense subset of its domain (cf. \cite[Corollary 13.14]{Willard}), both sides of~(\ref{Eq:CsystemsComparisionCommute}) must agree on all of $\delta(\cC_{X'})$ as desired.
\end{proof}

Our next Lemma shows that $\le_{\cC}$ extends to \Ccomp s:

\begin{lemma}\label{Csystem:leCextendsToCcomps}
If $\cC^\alpha$ and $\cC^\delta$ are two \Csys s with $\cC^\alpha\le_{\cC}\cC^\delta$, then we have
$\id_G\cup\psi^{\delta\alpha}\colon\delta G\ge\alpha G$.
\end{lemma}
\begin{proof}
We write $\psi$ for the map $\id_G\cup\psi^{\delta\alpha}$.
Since $\psi$ fixes $\Omega$, it remains to verify continuity of $\psi$.
For this, let any basic open set of $\alpha G$ be given; we may assume that it is of the form $\cO_{\alpha G}(X,O)$.
We claim that
\begin{align}\label{eq:preimageCheck}
\psi^{-1}(\cO_{\alpha G}(X,O))=\cO_{\delta G}(X,f_X^{-1}(O))
\end{align}
holds where $f_X\colon (\delta (\cC_X),\delta_X)\ge (\alpha (\cC_X),\alpha_X)$.
To see this we first note that both sides of~(\ref{eq:preimageCheck}) agree on $G$ due to $\alpha_X^{-1}(O)=\delta_X^{-1}(f_X^{-1}(O))$.
And second we note that the diagram
\begin{equation*}
\begin{tikzcd}
\cI^\delta\arrow[r,"\pi_X^\delta"]\arrow[d,"\psi^{\delta\alpha}"',two heads] & \delta(\cC_X)\arrow[d,"f_X",two heads]\\
\cI^\alpha\arrow[r,"\pi_X^\alpha"'] & \alpha(\cC_X)
\end{tikzcd}
\end{equation*}
commutes, resulting in
\begin{align*}
(\psi^{\delta\alpha})^{-1}\big( (\pi_X^\alpha)^{-1}(O)\big)=(\pi_X^\delta)^{-1}\big( f_X^{-1}(O)\big)
\end{align*}
which shows that both sides of~(\ref{eq:preimageCheck}) agree on $\cI^\delta$.
\end{proof}

\begin{definition}
If $\cC^\alpha$ and $\cC^\delta$ are two $\cC$-systems with both $\cC^\alpha\le_{\cC}\cC^\delta$ and $\cC^\delta\le_{\cC}\cC^\alpha$, then we say that $\cC^\alpha$ and $\cC^\delta$ are \emph{$\cC$-equivalent}.
\end{definition}

Using Lemma~\ref{Top:compactification:witnessBehaviour} it is not hard to show that

\begin{lemma}\label{Csystem:psiDeltaAlphaHomeo}
If $\cC^\alpha$ and $\cC^\delta$ are two $\cC$-equivalent $\cC$-systems, then both $\psi^{\delta\alpha}$ and $\psi^{\alpha\delta}$ are homeomorphisms and each other's inverse.\qed
\end{lemma}

\begin{corollary}\label{Csystem:CeqImpliesTopEq}
If $\cC^\alpha$ and $\cC^\delta$ are two $\cC$-equivalent $\cC$-systems, then $\alpha G$ and $\delta G$ are topologically equivalent, witnessed by the homeomorphism $\id_G\cup\psi^{\delta\alpha}$ and its inverse $\id_G\cup\psi^{\alpha\delta}$.
\end{corollary}
\begin{proof}
We combine Lemma~\ref{Csystem:leCextendsToCcomps} and Lemma~\ref{Csystem:psiDeltaAlphaHomeo}.
\end{proof}

\section{Critical vertex sets give rise to the coarsest \texorpdfstring{$\cC$}{C}-compactification}
\label{sec7}

Our aim in this section is to find the coarsest \Ccomp . Surprisingly, critical vertex sets will lead the way.
In Section~\ref{sec:EndsCritAsTangles} we constructed an inverse system $\{\finC{X},\eb_{X',X}\}$ (cf.~p.~\pageref{eb_X'Xdef}) giving the $\aleph_0$-tangles of $\St$, i.e. with $\invLim{}\finC{X}=\F=\Tt$.
We have seen in Proposition~\ref{Tangles:F(T,X)} that every $\aleph_0$-tangle of $\St$ induces, for every $X\in\cX$, a particular type of filter on the poset $\Pt{X}$ of all tame subsets of $\cC_X$: the up-closure (in $\Pt{X}$) either of a singleton $\{C\}\subseteq\cC_X$ or of the cofinite filter on $\cC_X(Y)$ for some critical $Y\subseteq X$.
With this in mind, we equip the sets $\finC{X}$ with a topology that turns their inverse system into a \Csys :

Given $X\in\cX$ we endow $\finC{X}=\cC_X\sqcup\crit(X)$ with the topology obtained by declaring as open in addition to the open sets of the discrete component space $\cC_X$, for all $Y\in\crit(X)$ and all cofinite subsets $\cC$ of $\cC_X(Y)$, the sets
\begin{align*}
\cO_{\finC{X}}(Y,\cC):=\cC\sqcup\{Y\}
\end{align*}
and taking the topology on $\finC{X}$ this generates.

\begin{lemma}\label{finCXareCompactHDtotallyDisc}
Every $\finC{X}$ is a finite Hausdorff compactification of $\cC_X$.\qed
\end{lemma}

\begin{lemma}
The maps $\eb_{X',X}\colon\finC{X'}\to\finC{X}$ are continuous.
\end{lemma}
\begin{proof}
For this, let any $\xi\in\finC{X'}$ be given together with a basic open neighbourhood $O$ of $\eb_{X',X}(\xi)$ in $\finC{X}$. We check two cases:

First, we suppose that $\xi$ is a component $C'$ of $G-X'$. 
Then $\eb_{X',X}$ sends the open neighbourhood $\{C'\}$ of $C'$ into $O$.

Second, we suppose that $\xi$ is a critical subset $Y$ of $X'$.
If $Y\notin\crit(X)$ then $\eb_{X',X}(Y)=C_X(Y)$ is the component of $G-X$ that includes $\bigcup\cC_{X'}(Y)$, so $\eb_{X',X}$ sends $\cO_{\finC{X'}}(Y,\cC_{X'}(Y))$ into $\{C_X(Y)\}\subseteq O$.
Otherwise $Y\in\crit(X)$ results in $Y=\eb_{X',X}(Y)$.
Then $O$ is of the form $\cO_{\finC{X}}(Y,\cC)$ for some cofinite subset $\cC$ of $\cC_X(Y)$. Hence $\cC':=\cC\cap\cC_{X'}$ is a cofinite subset of $\cC_{X'}(Y)$ with $\cb_{X',X}[\cC']\subseteq\cC$, so $\cO_{\finC{X'}}(Y,\cC')$ is an open neighbourhood of $Y$ which $\eb_{X',X}$ sends into $O$.
\end{proof}

Altogether we have shown that
\begin{proposition}
$\{\finC{X},\eb_{X',X}\}$ is a \Csys .\qed
\end{proposition}
\begin{notation}
We write $\minC$ for the \Csys\ $\{\finC{X},\eb_{X',X}\}$ and we write $\minG$ for the \Ccomp\ of $G$ which it induces by Theorem~\ref{Csystem:EveryCsystemInducesOmegaComp}.
\end{notation}
We obtain an analogue of Diestel's Theorem~\ref{DiestelsTangleCompWorks} for our $\minG$:
\begin{theorem}
Let $G$ be any graph.
\begin{enumerate}
\item $\minG$ is an $\Omega$-compactification of $G$ and $\minG\setminus G$ is totally disconnected.
\item If $G$ is locally finite and connected, then $\minG=\VG$ coincides with the Freudenthal compactification of $G$.
\end{enumerate}
\end{theorem}
\begin{proof}
The $\finC{X}$ are totally disconnected by Lemma~\ref{finCXareCompactHDtotallyDisc} and so is $\Gamma=\invLim{}\Gamma_X$.
\end{proof}

The next two lemmas are all we need to show that $\minC$ is the least \Csys :

\begin{lemma}\label{CsystemMinimalityInductionStep}
If $\cC^\alpha$ is a \Csys , then 
\begin{align}\label{Eq:distinctCritXDisjointClosures}
\overline{\alpha_X[\cC_X(Y)]}\cap\overline{\alpha_X[\cC_X(Y')]}=\emptyset
\end{align}
holds for all $X\in\cX$ and all distinct $Y,Y'\in\crit(X)$.
\end{lemma}
\begin{proof}
Let $X\in\cX$ and any two distinct $Y,Y'\in\crit(X)$ be given.
Without loss of generality we find some $x\in Y'\setminus Y$, and we set $X^-=X\setminus\{x\}$.
It is known (and not hard to verify) that a continuous map $h$ satisfies $h\big[\,\overline{A}\,\big]\subseteq\overline{h[A]}$ for each subset $A$ of its domain (cf.~\cite[Theorem 7.2]{Willard}). Thus we compute
\begin{align*}
\overline{\alpha_X[\cC_X(Y)]}\cap\overline{\alpha_X[\cC_X(Y')]}&\subseteq\overline{\fraka_{X,X^-}[\alpha_X[\cC_X(Y)]]}\cap\overline{\fraka_{X,X^-}[\alpha_X[\cC_X(Y')]]}\\
&\overset{(\ref{Eq:CsystemCommute})}{=}\overline{\alpha_{X^-}[\cb_{X,X^-}[\cC_X(Y)]]}\cap\overline{\alpha_{X^-}[\cb_{X,X^-}[\cC_X(Y')]]}\\
&=\overline{\alpha_{X^-}[\cC_X(Y)]}\cap\{\alpha_{X^-}(C_{X^-}(Y')\}\;.
\end{align*}
Write $C$ for $C_{X^-}(Y')$.
The point $\alpha_{X^-}(C)$ is isolated in $\alpha (\cC_{X^-})$ since $\{\alpha_{X^-}(C)\}$ is open.
Therefore, in order to verify (\ref{Eq:distinctCritXDisjointClosures}) for $Y$ and $Y'$ it suffices to show $C\notin\cC_X(Y)$.
The component $C$ meets $Y'$ in $x$ and hence is not a component of $G-X$.
In particular, it cannot be contained in $\cC_X(Y)\subseteq\cC_X\cap\cC_{X^-}$.
\end{proof}

\begin{lemma}\label{distinctCritXDisjointClosuresSuffices}
Let $X\in\cX$ be given with a Hausdorff compactification $(\alpha(\cC_X),\alpha_X)$ of $\cC_X$ satisfying (\ref{Eq:distinctCritXDisjointClosures})
for all distinct $Y,Y'\in\crit(X)$. Then $\finC{X}\le (\alpha(\cC_X),\alpha_X)$ holds.
\end{lemma}
\begin{proof}
The finite set $\alpha_X[\cC_X^-]$ (cf.~p.~\pageref{CXminus}) is closed in the \HD\ space $\alpha (\cC_X)$.
Moreover, since $\alpha_X[\cC_X]$ is open in $\alpha (\cC_X)$, 
we conclude that the clopen set $\alpha_X[\cC_X^-]$ avoids the closure of $\alpha_X[\cC_X(Y)]$ for all $Y\in\crit(X)$.
Therefore, (\ref{Eq:distinctCritXDisjointClosures}) gives rise to a finite partition of $\alpha (\cC_X)$ into closed sets:
\begin{align*}
\alpha (\cC_X)=\alpha_X[\cC_X^-]\sqcup\bigsqcup_{Y\in\crit(X)}\overline{\alpha_X[\cC_X(Y)]}\;.
\end{align*}
Let $f^-\colon\alpha_X[\cC_X^-]\to\cC_X^-$ send each $\alpha_X(C)$ to $C$.
This is continuous since $\alpha_X[\cC_X^-]$ carries the discrete subspace topology. 
For $Y\in\crit(X)$ we note that the closure of $\cC_X(Y)$ in $\finC{X}$ is the one-point \HDcomp\ of $\cC_X(Y)$, yielding some
\begin{align*}
f_Y\colon \big(\,\overline{\alpha_X[\cC_X(Y)]}\,,\,\alpha_X\rest\cC_X(Y)\,\big)\ge \overline{\cC_X(Y)}^{\finC{X}}.
\end{align*}
Since the domains of $f^-$ and the $f_Y$ form a finite partition of $\alpha(\cC_X)$ into closed sets, these continuous mappings combine to one continuous $f\colon \alpha(\cC_X)\to\finC{X}$ (cf. \cite[Theorem 18.3]{Munkres}, also known as `Pasting Lemma').
Then $f\colon (\alpha(\cC_X),\alpha_X)\ge \finC{X}$.
\end{proof}

\begin{theorem}\label{Csystem:leastAndGreatest}\label{minGquotientModG}
Let $G$ be any graph. The following hold up to $\cC$-equivalence:
\begin{enumerate}
\item $\minC$ is the least $\cC$-system with respect to $\le_{\cC}$.
\item $\cC^{\cU}$ is the greatest $\cC$-system with respect to $\le_{\cC}$.
\end{enumerate}
Moreover, the following hold up to topological equivalence:
\begin{enumerate}[resume]
\item $\minG$ is the coarsest \Ccomp .
\item $\modG$ is the finest \Ccomp .
\end{enumerate}
Therefore $\minG$ is a quotient of every \Ccomp\ which in turn is always a quotient of $\modG$. 
In particular, $\minG$ and $\modG/{\sim}$ are topologically equivalent.
\end{theorem}
\begin{proof}
(i). Let $\cC^\alpha$ be any \Csys . Lemmas~\ref{CsystemMinimalityInductionStep} and~\ref{distinctCritXDisjointClosuresSuffices} yield $\finC{X}\le (\alpha (\cC_X),\alpha_X)$ for all $X\in\cX$.
By Lemma~\ref{Csystem:leForCompsExtendsToleC} this implies $\cC^{\F}\le\cC^\alpha$.

(ii) is immediate from Lemma~\ref{Csystem:leForCompsExtendsToleC} and $\cU_X=\beta(\cC_X)$.

(iii) and (iv) follow from (i) and (ii), respectively, with Lemma~\ref{Csystem:leCextendsToCcomps}.

For the last two statements we combine (iii) and (iv) with Proposition~\ref{MagicProp}.
\end{proof}

\begin{theorem}\label{Ccomp:topologicalEquivalence}
Let $G$ be any graph. The following are equivalent:
\begin{enumerate}
\item $\minG$ and $\modG$ are topologically equivalent.
\item $\minC$ and $\cC^{\cU}$ are $\cC$-equivalent.
\item $G$ is \sol .
\item $\Omega=\F=\cU$.
\end{enumerate}
\end{theorem}
\begin{proof}
Both (iv)$\leftrightarrow$(iii) and (iii)$\to$(ii) are clear, whereas (ii)$\to$(i) holds by Corollary~\ref{Csystem:CeqImpliesTopEq}. Therefore, it suffices to show (i)$\to$(iii).

(i)$\to$(iii). 
Combining (i) with Theorem~\ref{minGquotientModG} yields $\modG\cong\modG/{\sim}$.
We assume for a contradiction that (iii) fails, witnessed by some $X\in\cX$ with $\cC_X$ infinite. 
Then $\crit(X)$ is non-empty and ${\sim}$ is non-trivial by Theorem~\ref{critX=critUF}, so Lemma~\ref{HDcompQuotient} yields
$(\modG/{\sim})\setminus\mathring{E}\lneq \modG\setminus\mathring{E}$
contradicting $\modG\cong\modG/{\sim}$ as desired.
\end{proof}

\begin{obs}
Let $G$ be any graph. There are four ways to describe $\minG$:
\begin{enumerate}
\item $\minG=\modG/{\sim}$ where ${\sim}$ can be described in terms of tangles.
\item $\minG=G\sqcup\Omega\sqcup\crit(G)$ involves only very basic combinatorics.
\item $\minG=G\sqcup\Tt$ is a tangle-type \comp .
\item $\minG=G\sqcup\Gamma$ is the coarsest \Ccomp .
\end{enumerate}
\end{obs}

\bibliographystyle{amsplain}
\bibliography{EndsTanglesCritBIB}

\end{document}